\documentclass[a4paper, reqno]{amsart}
\usepackage{graphicx}
\usepackage{amssymb}
\usepackage{epstopdf}
\usepackage{enumerate}
\usepackage{relsize}

\usepackage{graphicx, amssymb, mathrsfs, eucal}

\usepackage[vcentermath]{youngtab}

\usepackage{tikz}
\usetikzlibrary{cd}

\renewcommand{\k}{\Bbbk}

\newcommand{\N}{\mathbb{N}}
\newcommand{\Z}{\mathbb{Z}}

\def \B {\mathcal B}
\def \C {\mathcal C}
\def \F {\mathcal F}
\def \J {\mathcal J}
\def \M {\mathcal M}
\def \P {\mathcal P}

\def \T {\mathcal T}
\def \U {\mathcal U}
\def \V {\mathcal V}

\renewcommand{\b}{\mathtt{b}}
\renewcommand{\c}{\mathtt{c}}

\renewcommand{\i}{\mathtt{i}}
\renewcommand{\j}{\mathtt{j}}

\renewcommand{\r}{\mathtt{r}}
\newcommand{\s}{\mathtt{s}}

\def \Si{\mathfrak S}

\def \G{\Gamma}
\def \Sym{\mathrm{Sym}}
\def \L{\Lambda}

\def \op{\mathrm{op}}

\def \tr{\mathsf{tr}}
\def \tw{\mathsf{tw}}

\def \seq{\mathtt{seq}}
\def \Tab{\mathtt{Tab}}
\def \st{\mathtt{St}}

\newcommand{\tensor}{\otimes}

\newcommand{\bs}{\boldsymbol}

\newcommand{\<}{\langle}
\renewcommand{\>}{\rangle}

\newcommand{\Hom}{\mathrm{Hom}}
\newcommand{\End}{\mathrm{End}}

\newcommand{\Mat}{\mathrm{M}}
\newcommand{\VV}{\mathrm{V}}
\newcommand{\Vn}{\mathrm{V}_{\hspace{-1pt}n}}

\newcommand{\hs}[1]{\hspace{#1}}

\newcommand{\hm}[1]{\hspace{-.0#1cm}}
\newcommand{\hp}[1]{\hspace{.0#1cm}}

\newcommand{\set}[1]{[#1]}
\renewcommand{\rm}[1]{\mathrm{#1}}

\newcommand{\bm}[1]{\mathbf{#1}}

\mathchardef\hyphen="2D

\renewcommand{\-}{\hyphen\hyphen}

\newcommand{\Amod}{A\text{-mod}}
\newcommand{\modA}{\text{mod-}A}
\newcommand{\Bmod}{B\text{-mod}}
\newcommand{\modB}{\text{mod-}B}

\newcommand{\lmod}{\hyphen \mathrm{mod}}

\newcommand{\isoto}{\xrightarrow{\,\sim\,}}
\newcommand{\equi}{\stackrel{\sim}{\longrightarrow}}
\newcommand{\onto}{\twoheadrightarrow}

\newcommand{\ml}[1]{{\displaystyle \mathlarger{#1}}}

\newtheorem{theorem}{Theorem}[section]
\newtheorem{lemma}[theorem]{Lemma}
\newtheorem{proposition}[theorem]{Proposition}
\newtheorem{corollary}[theorem]{Corollary}
\theoremstyle{definition}
\newtheorem{definition}[theorem]{Definition}
\newtheorem{remark}[theorem]{Remark}
\newtheorem{example}[theorem]{Example}

\numberwithin{equation}{section}

\title[Cellularity of generalized Schur algebras]
{Cellularity of generalized Schur algebras\\ via Cauchy decomposition}

\author{Jonathan D.~Axtell}
\thanks{This paper was supported by the National Research Foundation of Korea (NRF) funded by the Ministry of Science (NRF-2017R1C1B5018384).} 
\address{Sungkyunkwan University, Suwon 16419, Republic of Korea} 
\email{jaxtell@skku.edu}

\begin{document}

\begin{abstract}
We describe a generalization of Hashimoto and Kurano's 
Cauchy filtration for divided powers algebras. 
This filtration is then used to provide   
a cellular structure for generalized Schur algebras
associated to an arbitrary cellular algebra, $A$. 
Applications to the cellularity of  
wreath product algebras $A\wr \Si_d$  
are also considered.  
\end{abstract}

\maketitle

\section{Introduction}

Let $\k$ be a noetherian integral domain and suppose $A$ is a cellular 
$\k$-algebra \cite{GL}. 
Then Geetha and Goodman \cite{GG}
showed that the wreath product algebra 
\[A\wr \Si_d  = A^{\tensor d} \rtimes \k\Si_d\] 
is cellular,  provided that all of the cell ideals of $A$ are cyclic. 
On the other hand, 
the generalized Schur algebras $S^A(n,d)$ were
defined by Evseev and Kleshchev   \cite{EK1, EK2} 
in order to prove the Turner double conjecture.
These algebras are related to 
wreath product algebras by a
generalized Schur-Weyl duality established in \cite{EK1}. 
\smallskip

In this paper, we describe a cellular structure for the 
generalized Schur algebra $S^A(n,d)$  
for an arbitrary cellular algebra $A$ 
and for all integers $n,d\geq 0$. 
This extends some results of 
Kleshchev and Muth \cite{KM1, KM2, KM3}. 
It follows, for example, from results of \cite{KM3}
that the algebra 
$S^A(n,d)$ is cellular  
for certain algebras $A$ which are both cellular and quasi-hereditary. 
We note that for such algebras, the cell ideals are automatically cyclic.
The method used in this paper, however, does not require any additional
assumptions on the cellular algebra. 
\smallskip

Our approach is motivated by that of  \cite{Krause1}, 
where Krause used the Cauchy decomposition of divided powers 
\cite{ABW, HK} 
to describe the highest weight structure of categories of 
strict polynomial functors. 
As Krause mentions, this leads to an alternate proof of 
the fact that classical Schur algebras $S^\k(n,d)$ are quasi-hereditary,  
which follows by a Morita equivalence.
As we will see, this approach can similarly be used  
 to describe cellular structure. 
\smallskip

We begin by constructing a generalized Cauchy filtration for the 
divided powers $\Gamma^d J$ of a given $\k$-module, $J$,  
which we assume is equipped with a filtration 
\begin{equation*}
 0 = J_1 \subset \cdots \subset J_r = J
\end{equation*}
such that $J_j/J_{j-1} \cong U_j\otimes_\k V_j$, 
for some free $\k$-modules $U_j, V_j$ of finite rank. 
Our first main
 result is a generalized Cauchy decomposition formula
(Theorem \ref{thm:gen_Cauchy}), 
which provides 
a filtration of $\Gamma^d J$ such that 
the associated graded object is a direct sum 
of modules of the  form 
\[\bigoplus_{{\bs \lambda} \in \bs \L} 
\U_{\bs \lambda} \tensor_\k \V_{\bs \lambda},\]
where 
$\U_{\bs \lambda}, \V_{\bs \lambda}$
are {\em generalized Weyl modules} defined in Section \ref{ss:gen_Weyl} 
and $\boldsymbol{\Lambda}$ denotes 
a set of $r$-multipartitions. 
\smallskip

The generalized Schur algebra $S^A(n,d)$ may be identified as the 
$d$-th divided power $\Gamma^d \Mat_n(A)$, 
where $\Mat_n(A)$ is the algebra of size $n$ matrices over $A$. 
We are thus able to use the above decomposition, 
together with K\"onig and Xi's characterization of cellular algebras in \cite{KX}, 
to prove our second main result  
(Theorem \ref{thm:cellular}) 
which shows that generalized Schur algebras are cellular. 
In Example \ref{ex:zig}, we describe a corresponding cellular basis 
explicitly for a particular case, $S^Z(1,2)$, where $Z$ 
is a zig-zag algebra 
(considered as an ordinary algebra rather than a superalgebra, 
as in \cite{KM3}).  
\smallskip

As a consequence of generalized Schur-Weyl duality, 
Corollary \ref{wreath} shows that the 
wreath product algebras $A\wr \Si_d$ are cellular for an arbitrary 
cellular algebra $A$.  
This provides an alternate proof of the 
main result in \cite{GG}, for the case where $A$ is cyclic cellular, 
and a more  
recent result of Green \cite{RGr}, for the general case 
where $A$ is an arbitrary cellular algebra.

\section{Preliminaries} 
Assume throughout that $\k$ is a commutative ring, 
unless mentioned otherwise.
The notation $\sharp$ is used for the cardinality of a set.

\subsection{Weights, partitions, and sequences} 
\label{ss:definition}
Write $\N$ and $\N_0$ to denote the sets of 
positive and nonnegative integers, respectively, 
with the usual total order. 
More generally, suppose that $\B$ is a countable 
totally ordered set which is bounded below. 
Any elements $a,b \in \B$ determine an interval  
\[ \set{a,b} := \{c\in \B \mid a \leq c \leq b\} \]
which is empty unless $a\leq b$.  
\smallskip

A {\em weight} ({\em on $\B$}) is a sequence 
of nonnegative integers  $\mu = (\mu_b)_{b\in \B}$ 
such that $\mu_b =0$ for almost all $b$. 
Let $\L(\B)$ denote the set of all weights on $\B$. 
A {\em partition} ({\em on $\B$}) is a weight $\lambda \in \L(\B)$ such that
\[b<c \ \text{ implies }\ \lambda_b \geq \lambda_c, 
\ \, ^\forall b,c\in \B.\]
The subset of partitions is denoted 
 $\L^+(\B) \subset \L(\B)$. 
The size of a weight $\mu$ is the integer 
$|\mu| := \sum_{b} \mu_b$. 
Let $\L_d(\B)$ denote the set of all weights 
of size $d$ and write  
\[\L^+_d(\B) := \L^+(\B) \cap \L_d(\B) \]
for each $d\in \N_0$. 
\smallskip

\begin{remark}\label{notation}
In this notation and elsewhere, 
we will use the convention of replacing 
an argument of the form  
 $\set{1,n}$ by ``$n$" for any $n\in \N_0$, 
so that for example 
$\Lambda(n)$ denotes the set 
$\Lambda(\set{1,n})$ of 
weights of the form 
$\mu=(\mu_1, \dots, \mu_n)$. 
\end{remark}
We also identify each set $\L(n)$ 
as a subset of $\L(\N)$ in the obvious way and write 
\[ l(\mu) := \mathrm{min}\{n\in \N_0\mid \mu \in \L(n)\} \]
to denote the 
 length 
of a weight $\mu \in \L(\N)$.
For example, the length $l(\lambda)$ 
of a partition 
$\lambda= (\lambda_1, \lambda_2, \dots)$ in $\L^+(\N)$ 
equals the number of positive parts, $\lambda_i\in \N$.

\begin{definition}\label{lex1}
Let $d\in \N_0$. 
Recall that the {\em lexicographic ordering} on $\L_d(\N)$ 
is the total order defined by setting 
$\lambda \leq \mu$ if 
$\lambda_j \leq \mu_j$ whenever   
$\lambda_i=\mu_i$ for all $i<j$.   
We use the notation $\preceq$ to denote the restriction 
of $\leq$ to the subset 
  $\Lambda^+_d(\N)$ of partitions of size $d$. 
\end{definition}
\smallskip

Now fix $d\in \N$, 
and write $\seq^{d}(\B)$ to denote the set of all functions 
\[\b: \set{1,d} \to \B.\] 
We identify  $\seq^{d}(\B)$ with $\B^d$ 
by setting $\mathtt b = (b_1, \dots, b_d)$,  
with $b_i= \mathtt b(i)$ for all $i\in \set{1,d}$.
The symmetric group $\Si_d$ of permutations of $\set{1,d}$ 
acts on $\seq^{d}(\B)$ from the right via composition. 
We write $\mathtt b \sim \mathtt c$ 
if there exists $\sigma \in \Si_d$ with
$\mathtt c = \mathtt b \sigma$. 
\smallskip

The {\em weight} of a sequence $\b\in \seq^d(\B)$ 
is the element of $\Lambda_d(\B)$ defined by 
\begin{equation*}
\mu(\b) := (\mu_{c})_{c\in \B}, 
\quad \text{where}\quad
\mu_{c} = \sharp \{i  \mid b_i =c\}\ \ ^\forall c\in \B.
\end{equation*}
We note the following elementary result. 

\begin{lemma}\label{mu}
The map 
$\mu: \seq^{d}(\B) \to  \L_d(\B)$, 
sending $\mathtt b \mapsto \mu({\mathtt b})$, 
induces a bijection:
$\seq^d(\B)/\Si_d \ 
 \simeq\ \L_d(\B).$
\end{lemma}

\begin{proof}
We may assume that $\B$ is nonempty. 
Since $\B$ is bounded below, 
it is possible to write the elements explicity in the form 
\begin{equation}\label{setB}
\B = \{b^{\B}_1 < b^{\B}_2 < \dots\ \}.
\end{equation} 
To show that the map 
$\mathtt b \mapsto \mu({\mathtt b})$ 
 is surjective, 
note that a right inverse is given by   
\begin{equation*}
 \L_d(\B) \to \seq^{d}(\B): \ 
\mu \mapsto \mathtt b_\mu := 
(b^\B_1, \dots, b^\B_1, b^\B_2, \dots, b^\B_2, \dots) 
\end{equation*} 
where $b^\B_1$ occurs with multiplicity 
$\mu_{b^\B_1}$, etc. 
Finally, it is easy to see that 
$\mathtt b \sim \mathtt c$ if and only if 
$\mu(\mathtt b)= \mu(\mathtt c)$, 
which completes the proof. 
\end{proof}
\smallskip

Suppose more generally that $\B_1, \dots, \B_r$ 
is a collection of bounded below, totally ordered sets. 
We again consider the product 
$\B=\B_1\times \dots\times \B_r$
as a bounded below, totally ordered set 
via the lexicographic ordering. 

The symmetric group $\Si_d$ 
acts diagonally on the following product 
\begin{equation*}
 \seq^d(\B_1, \dots, \B_r) 
:= \seq^d(\B_1) \times \dots \times  \seq^d(\B_r). 
\end{equation*}
Notice that the bijection 
\[\theta: \seq^d(\B_1, \dots, \B_r) 
\simeq \seq^d(\B)\]
defined by  
\[\theta(\b^{(1)}, \dots, \b^{(r)}):\ 
i \mapsto (b^{(1)}_i, \dots, b^{(r)}_i), \quad ^\forall i\in \set{1,d},\]
is $\Si_d$-equivariant. 
It thus follows as an immediate consequence of Lemma \ref{mu}
that there is a bijection 
\begin{equation}\label{mu2}
\seq^d(\B_1, \dots, \B_r)/\Si_d
\, \simeq \, \L_d(\B), 
\end{equation}
where $\seq^d(\B_1, \dots, \B_r)/\Si_d$ 
denotes the set of diagonal $\Si_d$-orbits.

\subsection{Multipartitions} 
Suppose $d\in \N_0$ and let $\B_1, \dots, \B_r$ 
be as above. 
Then we use the following notation for the product
\[ \L^+(\B_1, \dots, \B_r) 
:= \, \L^+(\B_1) \times \dots \times \L^+(\B_r). \]
whose elements are called {\em $r$-multipartions} and denoted  
${\bs \lambda} = (\lambda^{(1)}, \dots, \lambda^{(r)})$. 
The {\em weight} of an $r$-multipartion $\bs \lambda$ 
is the element of $\L(r)$ defined by 
\[ |\bs \lambda| := (|\lambda^{(1)}|, \dots, |\lambda^{(r)}|).\]
We call $||\bs \lambda|| := \sum |\lambda^{(j)}|$ 
the {\em total weight} (or {\em size}) of $\bs \lambda$. 
\smallskip

Given $\mu\in \Lambda(r)$ and $d\in \N_0$, we write 
\[\L^+_{\mu}(\B_1, \dots, \B_r) \, := \, 
\L^+_{\mu_1}(\B_1) \times \dots \times \L_{\mu_r}^+(\B_r)\]  
and
\[\L^+_d(\B_1, \dots, \B_r) 
\, := \bigsqcup_{\nu \in \L_d(r)} \L^+_{\nu}(\B_1, \dots, \B_r) \]
to denote the subset of $r$-multipartions of weight $\mu$,  
resp.~total weight $d$. 
\smallskip

In the special case where $\B_j=\N$ for $j\in \set{1,r}$, note that  
\[\L^+(\N, \dots, \N) = \L^+(\N)^r.\] 
We then use the following notation 
\[\L^+_d(\N)^r
:= \L_d^+(\N, \dots, \N), \qquad 
\L^+_{\mu}(\N)^r 
:=\L_\mu^+(\N, \dots, \N)\]
for $d\in \N_0$ and $\mu\in \L_r(d)$, respectively. 
\smallskip

The next definition describes a total order on the set of $r$-multipartitions 
of a fixed total weight. 

\begin{definition}\label{lex2}
Suppose $d,r\in \N$. 
Then $\Lambda_d^+(\N)^r$ has a total order $\preceq$ defined 
as follows. 
For $r$-multipartitions 
 $\bs\mu, \bs\lambda \in \L^{+}_\nu(\N)$ 
of weight $\nu\in \Lambda_d(r)$, 
we set  
$\bs \lambda \preceq \bs \mu$ if 
\[\lambda^{(j)} \preceq \mu^{(j)},\text{ whenever } 
\lambda^{(i)} = \mu^{(i)} \text{ for all } i<j.\] 
We then extend $\preceq$ to all of 
$\L^{+}_d(\N)^r$ 
by setting $\bs\lambda \prec \bs \mu$
whenever $|\bs\lambda| < |\bs \mu|$ in 
the lexicographic ordering on $\L_d(r)$. 
\end{definition} 

Suppose $n_1, \dots, n_r\in \N_0$ and $d\in \N$. 
Recalling the notation from Remark \ref{notation}, 
we identify the set of $r$-multipartions 
\[ \L^+(n_1, \dots, n_r)
:= \L^+(\set{1,n_1}, \dots, \set{1,n_r}) \] 
as a subset of $\L^+(\N)^r$ 
and view $\preceq$ as a total order on 
$\L^+_d(n_1, \dots, n_r)$ by restriction.

\subsection{Finitely generated projective modules}
Let $\M_\k$ denote the category of all $\k$-modules 
and $\k$-linear maps.  The full subcategory of finitely generated 
projective $\k$-modules is denoted $\P_\k$. 

Given $M,N \in \M_\k$, we write $M\tensor N= M\tensor_\k N$ 
and $\Hom(M,N) = \Hom_\k(M,N)$. 
Also write $\End(M)$ to denote the $\k$-algebra $\Hom(M,M)$. 
If $M\in \P_\k$, we let $M^\vee= \Hom(M,\k)$ denote the $\k$-linear dual. 
For any $M, M',N, N'\in \P_\k$, there is an isomorphism 
\begin{equation}\label{eq:isom}
\Hom(M\tensor N, M'\tensor N') \cong \Hom(M, M')\tensor \Hom(N, N')
\end{equation}
which is natural with respect to composition.

\subsection{Divided and symmetric powers}
Let $d\in \N$. 
Given $M\in \P_\k$, there is a right action of 
the symmetric group $\Si_d$ 
on the tensor power $M^{\tensor d}$ given by permuting tensor factors. 
We define the {\em $d$-th divided power} of $M$ 
to be the invariant submodule 
\[ \G^dM := (M^{\tensor d})^{\Si_d}.\] 
Similarly, the coinvariant module is denoted 
\[ \Sym_d M := (M^{\tensor d})_{\Si_d}\] 
and called the {\em $d$-th symmetric power} of $M$.
It follows by definition that 
\begin{equation}\label{symmetric}
\Gamma^d(M)^\vee \cong \mathrm{Sym}_d(M^\vee).
\end{equation}
We also set $\Gamma^0M= \mathrm{Sym}_0M= \k$. 

Note that the isomorphism (\ref{symmetric}) is usually taken  
as the definition of $\Gamma^d M$ (cf. \cite{ABW}), 
while we have used the equivalent definition 
from \cite{Krause1} in terms of symmetric tensors.

\subsection{The divided powers  algebra}
The category $\M_\k$ (resp.~$\P_\k$) is a symmetric monoidal category with 
symmetry isomorphism  
\begin{equation}\label{symmetry}
\tw: M\tensor N \xrightarrow{\sim} N\tensor M
\end{equation}
defined by  $x\tensor y \mapsto y\tensor x$, for all $x\in M, y\in N$.
\smallskip

Suppose $M\in \P_\k$.  
Then 
\[\Gamma (M) := \bigoplus_{d\in \N_0} \Gamma^dM\] 
is an ($\N_0$-graded) commutative algebra 
called the {\em  divided powers algebra},
with multiplication defined on homogeneous components 
via the shuffle product: 
for $x \in \Gamma^{d} M$ and $y \in \Gamma^{e}M$, 
define 
\begin{align*}
x \ast y := \sum_{\sigma\in 
\Si_{d+ e}^{d, e}} 
(x\tensor y)\sigma 
\end{align*}
where $\Si_{d+e}^{d,e}$ is the quotient group
 $\Si_{d+e}/\Si_{d}\times \Si_{e}$.
For example, we have   
$x^{\tensor d} \ast x^{\tensor e} = 
 \binom{d+e}{d}\, 
x^{\tensor(d+e)}$ 
for any $x \in M$. 
\smallskip

There is also a comultiplication, 
 $\Delta: \G (M) \to \G (M) \tensor \G (M)$, which is
the $\N_0$-homogenous map 
whose graded components 
\[\Delta: \Gamma^d M 
\to  \Gamma^{d-c}M \tensor \Gamma^{c}M\] 
are defined as the inclusions 
\[(M^{\tensor d})^{\Si_d} 
\hookrightarrow (M^{\tensor d})^{\Si_{d-c} \times \Si_{c}}\]
induced by the embeddings 
$\Si_{d-c}\times \Si_{c}\hookrightarrow \Si_d$,
for $c\in \set{0,d}$. 
These maps, together with the unit, 
$\k=\Gamma^0M \hookrightarrow \Gamma (M)$, 
and the counit, 
$\Gamma(M) \onto \Gamma^0M$
(projection onto  degree 0),
make $\Gamma (M)$ into a bialgebra.

\subsection{Decompositions}
The {\em symmetric algebra}  $S(M)$ is defined as 
the free commutative $\k$-algebra
generated by $M$ and 
has a decomposition 
 \[ S(M) = \bigoplus_{d\in \N_0} \mathrm{Sym}_d M. \] 
It follows that $S(-)$ defines a functor from $\P_\k$ 
to the category of all commutative 
$\k$-algebras, which preserves coproducts. 
Hence $S(M)\tensor S(N) \cong  S(M\oplus N)$, 
and by the duality \eqref{symmetric} there is an isomorphism
\begin{equation}\label{eq:expon}
\G (M) \tensor \G (N) \simeq \G(M\oplus N).
\end{equation}  
The isomorphism (\ref{eq:expon}) is given explicitly 
by restricting the multiplication map $x\tensor y \mapsto x \ast y$, 
where $\G(M)$, $\G( N)$ are 
considered as subalgebras of $\G(M\oplus N)$. 
It follows that for each $d\in \N_0$ there is  
a decomposition 
\begin{equation}\label{eq:expon2}
\Gamma^d(M\oplus N) = 
\bigoplus_{0\leq c\leq d} \Gamma^c(M)\ast \Gamma^{d-c}(N) 
\end{equation}
where
$\G^c (M)\ast\G^{d-c}(N)$ 
denotes the image of 
$\Gamma^c (M)\tensor \Gamma^{d-c}(N)$ 
 under (\ref{eq:expon}).
\smallskip

Note that $\Gamma^d \k \cong \k$ for all $d\in \N_0$.  
Thus, given a free $\k$-module $V$ of finite rank, 
it follows by induction from  (\ref{eq:expon2}) 
 that the  divided power $\Gamma^d V$ 
is again a free $\k$-module of finite rank.  
For example, suppose $V$ has a finite ordered  
$\k$-basis $\{x_b\}_{b\in \B}$. 
Then $\Gamma^d V$ 
has the following $\k$-basis 
\begin{equation}\label{basis}
\Big\{
x^{\mu} := \prod_{b\in \B}x_b^{\tensor \mu_b}  \ \ml{\ml{|}}\  \mu \in \L_d(\B)
\Big\} 
\end{equation}
where the product denotes multiplication in $\Gamma(V)$. 
\smallskip

The basis \eqref{basis}  can also be parameterized 
by elements of $\seq^d(\B)$. 
First notice that the the tensor power $V^{\tensor d}$ has 
the following basis 
\[ \{x_{\tensor\b} := x_{b_1}\tensor \dots \tensor x_{b_d}    \, \Big|\ 
\b \in \seq^d(\B)\}. \]
Given $\b\in \seq^d(\B)$, 
we then define 
$x_\b :=\sum_{{\tt b \sim c}} 
x_{\tensor{\tt c}}$. 
Notice that $x_\b = x_{\mu(\tt b)}$.  
It then follows from Lemma \ref{mu} that the set 
\begin{equation}\label{basis2}
\{x_\b \mid \b \in \seq^d(\B)/\Si_d\}\end{equation}  
is also a basis of $\Gamma^d V$,  
indexed by any complete set of orbit representatives. 
\smallskip

\subsection{Polynomial functors}\label{ss:functor}
We recall the definitions of some well known 
polynomial 
endofunctors on the category $\P_\k$ 
along with their associated natural transformations. 
\smallskip

Let $d\in \N_0$. Then recall the functor $\tensor^d: \P_\k \to \P_\k$ sending   
$M \mapsto M^{\tensor d}$,  whose action on morphisms is defined by 
\[ \tensor^d_{M,N}(\varphi) :=\,  
 \varphi \tensor \cdots \tensor \varphi: 
M^{\tensor d} \to N^{\tensor d} \] 
for any $\varphi\in \Hom(M,N)$.
\smallskip

It follows easily from \eqref{eq:expon2} 
that the divided power $\G^d M$
of a finitely-generated, projective $\k$-module $M\in \P_\k$ 
is again finitely-generated and projective. 
This yields a functor 
$\G^d: \P_\k \to \P_\k$ 
which is 
a subfunctor of $\tensor^d$.  
In particular, the 
action of $\G^d$ on morphisms 
is defined by restriction  
\[\G^d_{M,N}(\varphi) 
:= (\varphi^{\tensor d})|_{\G^d M}:
\Gamma^dM \to \Gamma^d N\]
for any $\varphi\in \Hom(M,N)$. 
\smallskip

Now let $S,T:\P_\k\to \P_\k$ 
be an arbitrary pair of functors. 
Then the tensor product $-\tensor - $ 
induces the following bifunctors 
\[S\boxtimes T, \ \  
T(-\tensor-)
\ :\  \P_\k\times \P_\k \to \P_\k\]
which are respectively defined by 
\[S\boxtimes T := (-\tensor-)\circ (S\times T), \ \qquad 
T(-\tensor-) := T\circ (-\tensor -).\] 
We also have the 
``object-wise" tensor product 
 $S\tensor T: \P_\k \to \P_\k$ defined by 
\begin{equation}\label{tensor}
S\tensor T := (S\boxtimes T)\circ \delta 
\end{equation}
where $\delta:\P_\k\to \P_\k\times \P_\k$ denotes the 
diagonal embedding: $M\mapsto (M,M)$.
\smallskip

Now suppose $M, N\in \P_\k$.  
As in \cite{Krause1}, define 
$\psi^d=\psi^d(M,N)$ to be 
the unique map which makes the  
following square commute: 
\begin{equation}\label{commute}
\begin{tikzcd}[row sep=large]
\Gamma^d M \tensor \Gamma^d N  
 \ar[d, tail ] \ar[r, "\psi^d"] 
& \Gamma^d(M\tensor N)
\ar[to=Z, d, tail, ]\\
M^{\tensor d} \tensor N^{\tensor d} 
\ar[r, "\sim"]
 &
(M \tensor N)^{\tensor d} 
\end{tikzcd}
\end{equation}
The following lemma is easy to check. 
\smallskip

\begin{lemma}\label{psi}
\begin{enumerate}
\item 
The maps $\psi^d(M,N)$ 
form a natural transformation of bifunctors  
\[\psi^d: \Gamma^d\boxtimes \Gamma^d 
\to \Gamma^d(-\tensor-).\]
\item
If $M,N\in \P_\k$, then the following 
diagram commutes  
\begin{equation*}
\begin{tikzcd}[column sep=huge] 
\Gamma^d M \tensor \Gamma^d N 
\ar[r, "{\psi^d(M,N)}"]  \ar[d, "\tw"' ]
 & 
\Gamma^d(M\tensor N) 
\ar[d, "\Gamma^d(\tw)"] 
\\
\Gamma^d N \tensor \Gamma^d M 
\ar[r, "{\psi^d(N,M)}"]  
 & 
\Gamma^d(N\tensor M) 
\end{tikzcd}
\end{equation*}
where $\tw$ permutes tensor factors 
as in (\ref{symmetry}). 
\end{enumerate}
\end{lemma}

\section{Generalized Schur Algebras} 
After recalling the definition of generalized Schur algebras \cite{EK1} 
associated to a $\k$-algebra $A$, 
we introduce corresponding 
standard homomorphisms between 
certain modules of divided powers.

\subsection{Associative $\k$-algebras}
Suppose that $R,S$ are associative algebras 
in the category $\M_\k$.  
Recall that the tensor product $R\tensor S$ 
is the algebra in $\M_\k$ with 
multiplication $m_{R\tensor S}$ defined by
\[ R\tensor S \tensor R \tensor S
\xrightarrow{\, 1\tensor \tw \tensor 1\, }
R\tensor R \tensor S \tensor S
\xrightarrow{\, m_R \tensor m_S\, }
R\tensor S. \]
Given $d\in \N$, the tensor power 
$R^{\tensor d}$ is an associative algebra in $\M_\k$ in a similar way. 
If $R$ is unital, then $R^{\tensor d}$ has unit $1_R^{\tensor d}$. 
\smallskip

In the remainder, the term {\em $\k$-algebra}
will always refer to a unital, associative algebra 
in the category $\P_\k$. 
Let $A\in \P_\k$ be a $\k$-algebra.   
Then $\Amod$ (resp.~$\modA$) denotes the subcategory of $\P_\k$
consisting of all left (right) $A$-modules, $M\in \P_\k$, 
and $A$-module homomorphisms. 
Write $\Hom_A(M,N)\in \P_\k$ to denote the 
set of all $A$-homomorphisms from $M$ to $N$  
for  $M,N \in \Amod$ (resp.~$\modA$).
We also write $\rho_M:A\tensor M \to A$ 
(resp.~$\rho_M:M\tensor A \to A$) to denote the 
induced linear map corresponding to a left
(right) $A$-module. 
 \smallskip

If $M\in \Amod$ (resp.~$\modA$) and $N\in \Bmod$ (resp.~$\modB$), 
the tensor product 
$M\tensor N$ is a left (resp.~right) $A\tensor B$-module, 
with corresponding module map: $\rho_{M\tensor N} = (\rho_M\tensor \rho_N)\circ(1\tensor \tau\tensor 1)$.

\subsection{The algebra $\Gamma^d A$}

Suppose $A$ is a $\k$-algebra.  Then $\Gamma^d A$ is a $\k$-algebra 
with multiplication $m_{\Gamma^dA}$ defined via the composition 
\[ \Gamma^d A\tensor \Gamma^d A
\xrightarrow{\psi^d} 
\Gamma^d(A\tensor A) 
\xrightarrow{\Gamma^d(m_A)}
\Gamma^d A, \] 
where the second map denotes  
the functorial action of $\Gamma^d$ on $m_A$. 
It follows that $\Gamma^d A$ is a unital subalgebra 
of $A^{\tensor d}$.

\begin{example}[The Schur algebra]

Suppose $n\in \N$, 
and let $\Mat_n(\k)$ 
denote the algebra 
of all $n\times n$-matrices in $\k$. 
Then  $\Gamma^d \Mat_n(\k)$ is isomorphic to 
the classical {\em Schur algebra}, 
$S(n,d)$, 
defined by Green \cite[Theorem 2.6c]{Green}. 
We view this isomorphism as an identification. 
\end{example}

We now have two distinct multiplications on  
the direct sum 
$\Gamma(A) = \bigoplus_{d \in \N} \Gamma^d A$. 
In order to distinguish them, we sometimes refer to the shuffle product 
\[\nabla: \Gamma^{d} A \tensor \Gamma^{e} A
\to \Gamma^{d+e} A: \, 
x\tensor y \mapsto x\ast y \] 
as {\em outer} multiplication in $\Gamma(A)$, 
while {\em inner} multiplication 
refers to the map 
defined as multiplication in $\G^d A$ 
on diagonal components 
\[ m_{\Gamma^d A}: 
\Gamma^{d} A \tensor \Gamma^{d} A\to \Gamma^d A:\, 
x\tensor y \mapsto x y \] 
and then extended by zero to 
other components.

\subsection{Generalized Schur algebras} \label{ss:Schur}

Given a  $\k$-algebra $A$, 
write $\Mat_n(A)$ for the algebra  
of $n\times n$-matrices in $A$. 
We   identify $\Mat_n(A)$ with  $\Mat_n(\k) \tensor A$ via 
\[\Mat_n(A)\, \xrightarrow{\,\sim\,}\,  \Mat_n(\k) \tensor A:\, 
(a_{ij})\mapsto \sum_{i,j}  
E_{ij} \tensor a_{ij},\] 
where $E_{ij}$ 
are elementary matrices in $\Mat_n(\k)$.  
Next, suppose $V$ is any left (resp.~right) $\Mat_n(\k)$-module, 
and let $M\in \Amod$ ($\modA$). 
Then write $V(M) := V\tensor M$ to denote the corresponding 
$\Mat_n(A)$-module. 

\begin{definition}
Suppose $A$ is an algebra, 
and let $n\in\N$, $d\in \N_0$. 
Then the {\em generalized Schur algebra} 
$S^A(n,d)$ is the algebra $\Gamma^d \Mat_n(A)$.
\end{definition}

Using the notation of \cite{EK1}, 
notice that $\Mat_n$ is spanned by the elements 
$\xi^a_{i,j} := E_{ij} \tensor a$,
for all  
$a\in A$ and $i,j \in \set{1,n}$.
Now suppose that $A$ is free as a $\k$-module with 
finite ordered basis $\{x_b\}_{b\in\B}$. 
Then $\Mat_n(A)$ has a corresponding basis 
\[\{\xi_{i,j,b} := \xi^{x_b}_{i,j} \mid i,j\in \set{1,n},\, b\in \B \}.\]
We view $\Mat_n(\k)$ as a subalgebra of $\Mat_n(A)$ 
by identifying $E_{ij} = \xi^{1}_{i,j}$. 
Notice that the classical Schur algebra $S(n,d)$ is thus a (unital) subalgebra of $S^A(n,d)$. 
\smallskip 

For each triple $(\i, \j, \b) \in \seq^d(n,n,\B)$, 
there is a corresponding element of $S^A(n,d)$ denoted  by
\[\xi_{\i,\j, \b}:= 
\sum_{(\i,\j, \b) \sim ({\tt r},{\tt s}, \c)}
\xi_{r_1,s_1, c_1}
\tensor \cdots \tensor 
\xi_{r_d,s_d, c_d},\]
where the sum is over all triples $ ({\tt r},{\tt s}, \c) $ in the same 
diagonal $\Si_d$-orbit as $(\i,\j, \b)$. 
It thus follows from \eqref{mu2}, \eqref{basis} and \eqref{basis2} 
that the set  
\[ \{\xi_{\i,\j,\b} \mid
(\i,\j,\b) \in \seq^d(n,n,\B)/\Si_d\}\]
forms a basis of of $S^A(n,d)$. 
In a similar way, the subalgebra $S(n,d)$ 
has a basis given by
\[\{\xi_{\i, \j} := 
\sum_{(\i,\j) \sim (\r,\s)} 
\xi^{1}_{r_1,s_1}
\tensor \cdots \tensor 
\xi^{1}_{r_d,s_d}\mid
(\i, \j) \in \seq^d(n,n)/\Si_d\}. \]
For each weight $\mu \in \L_d(n)$, we write 
\[\xi_\mu := \xi_{\i_\mu, \i_\mu}\]
to denote the corresponding idempotent 
in $S(n,d) \subset S^A(n,d)$.

\subsection{Standard homomorphisms}
Let us fix an algebra $A$ throughout the remainder of the section. 
Given $M\in \Amod$, 
it follows from (\ref{commute}) that $\Gamma^d M$ is 
a left $\Gamma^d A$-module 
with module map $\rho_{\Gamma^d M}$ 
determined by the composition 
\[ \Gamma^d A\tensor \Gamma^d M 
\xrightarrow{\psi^d} 
\Gamma^d(A\tensor M) 
\xrightarrow{\Gamma^d(\rho_M)}
\Gamma^d(M), \]
where the second map 
denotes the functorial action of $\Gamma^d$ 
on $\rho_M$.

\begin{lemma}\label{lem:hom}
Suppose $M,N\in \Amod$, and let $\varphi: M\to N$ 
be an $A$-module homomorphism.  Then the functorial 
map 
\[\Gamma^d(\varphi): \Gamma^d M \to \Gamma^d N\]
is a homomorphism of $\Gamma^d A$-modules. 
Moreover, if $\varphi$ is injective (resp.~surjective) 
then so is $\Gamma^d(\varphi)$. 
\end{lemma}

\begin{proof}
The map 
$\varphi^{\tensor d}: M^{\tensor d} \to N^{\tensor d}$ 
is a homomorphism of $A^{\tensor d}$-modules, 
and if $\varphi$ is injective (resp.~surjective) 
then so is $\varphi^{\tensor d}$. 
The statements for $\Gamma^d(\varphi)$ follow by 
restriction. 
\end{proof}

Suppose $d, e\in \N_0$ and $M, N \in  \Amod$. 
Notice that the homogeneous component
of comultiplication 
\begin{equation}\label{comult}
\Delta: \Gamma^{d+e} A 
\rightarrow \Gamma^{d} A \tensor \Gamma^{e} A
\end{equation}
is an injective (unital) map of $\k$-algebras. 
It follows that  
$\Gamma^{d} M \tensor \Gamma^{e} N$ 
has a corresponding $\Gamma^d A$-module structure, defined 
by restriction 
along \eqref{comult}. 
In the particular case $M=N$, we note that each of the following maps 
is a $\Gamma^d A$-module homomorphism:  
\begin{align}\label{eq:homom}
\Delta: \Gamma^{d+e} M
\to \Gamma^{d} M \tensor \Gamma^{e} M,
\qquad
& \qquad
\nabla:\Gamma^{d} M \tensor \Gamma^{e} M
\to \Gamma^{d+e} M, \nonumber
\\[.1cm]
\tw: \Gamma^{d} M \tensor \Gamma^{e} M\, 
&\xrightarrow{\sim}\, \Gamma^{e} M \tensor \Gamma^{d} M, 
\end{align}
where $\nabla$ (resp.~$\Delta$) are components of 
(co)multiplication in the bialgebra $\Gamma(M)$. 
Setting $A=\k$ then gives the following. 
\smallskip

\begin{lemma}\label{lem:natural}
Let $d, e\in \N$.  
Then there are natural transformations 
\begin{align*}
\Delta: \Gamma^{d+e} \to \Gamma^{d}  \tensor \Gamma^{e},
\qquad
& \qquad
\nabla:\Gamma^{d}  \tensor \Gamma^{e} 
\to \Gamma^{d+e} 
\end{align*}
of functors $\P_\k\to \P_\k$ 
induced by setting $\Delta(M)$ (resp.~$\nabla(M)$) 
equal to (co)multiplication in $\Gamma(M)$, 
for each $M\in \P_\k$.
\end{lemma}

Now suppose $r\in \N$ 
and  $\mu \in \L(r)$. 
Given $M, N_1, \dots, N_r \in \P_\k$, we write 
\[\Gamma^{(\mu)}(N_1, \dots, N_r) :=\,  
\Gamma^{\mu_1}N_1 \tensor \cdots \tensor \Gamma^{\mu_r} N_r\] 
and set 
\[ \Gamma^\mu M := 
\Gamma^{(\mu)}(M, \dots, M). \] 
If $M_1, \dots, M_r \in \Amod$, then
we consider 
$\Gamma^{(\mu)}(M_1, \dots, M_r)$
as a left 
$\Gamma^d A$-module by restriction along the corresponding  
inclusion, 
$\Delta: \Gamma^d A \to \Gamma^\mu A$, 
of $\k$-algebras. 
\smallskip

Suppose that $\gamma= (\gamma_{ij})\in \L_d(\N\times \N)$ 
is a (semi-infinite) matrix whose entries sum to $d$. 
Then let 
 $\lambda, \mu \in \L_d(\N)$  be weights such that 
$\lambda_i= \sum_j \gamma_{ij}$ and 
$\mu_j= \sum_i \gamma_{ij}$ for all $i,j\in\N$. 
Slightly abusing notation, for a given $N\in \P_\k$, 
we also write $\gamma = \gamma(N)$ to denote the corresponding 
{\em standard homomorphism:} 
\[\gamma: \Gamma^\mu N \rightarrow \Gamma^{\lambda}N\]
defined by the composition
\begin{align*}
\bigotimes_j\Gamma^{\mu_j} N
\xrightarrow{ \Delta \tensor \dots \tensor \Delta}
\bigotimes_i
\bigotimes_j \Gamma^{\gamma_{ij}} N 
\isoto
\bigotimes_j 
\bigotimes_i \Gamma^{\gamma_{ij}}N 
\xrightarrow{\nabla\tensor \dots \tensor \nabla}
\bigotimes_i \Gamma^{\lambda_i}N, 
\end{align*}
where each $\nabla$ (resp.~$\Delta$) denotes an 
appropriate component of 
(co)multiplication in the bialgebra $\Gamma(N)$, 
and where the second map rearranges the tensor factors. 
\smallskip

If $M\in  \Amod$, then it follows from 
(\ref{eq:homom}) that 
$\gamma(M): \Gamma^\mu M \to \Gamma^\lambda M$ 
is a  homomorphism of $\Gamma^d A$-modules. 
In the same way, we obtain  
homomorphisms of $S^A(n,d)$-modules 
corresponding to any given  $M \in \Mat_n(A) \lmod$. 
\smallskip

\subsection{Quotient modules} 
\label{ss:quotient}

Suppose $M\in \P_\k$.
Then we write $\< L\> \subset M^{\tensor d}$ 
to denote the  $\Si_d$-submodule generated by  
a subset $L \subset M^{\tensor d}$. 
For example if
 $L_1, \dots, L_d \subset M$ are 
 $\k$-submodules
and $L= L_1 \tensor \cdots \tensor L_d$, then 
\[\< L \> = \sum_{\sigma \in \Si_d} 
L_{1\sigma} \tensor \cdots \tensor L_{d\sigma}, \]
where $i\sigma := \sigma^{-1}(i)$ denotes the right action of 
$\sigma$ on $i\in \set{1,d}$. 
\smallskip

Now suppose   
$M=N\oplus N'$ for some $\k$-submodules $N,N' \subset N$.   
Then notice that there is a corresponding decomposition 
\[ M^{\tensor d}\, 
=\, (N')^{\tensor d} \oplus
\<N\tensor M^{\tensor d-1}\>, \]
which is a direct sum of $\Si_d$-submodules. 
Taking $\Si_d$-invariants on both sides results in 
the  decomposition 
\begin{equation}\label{invariant}
\Gamma^d M 
= \Gamma^d(N') \oplus 
\<N \tensor M^{\tensor d-1}\>^{\Si_d}
\end{equation}
into $\k$-submodules. 
The decomposition \eqref{invariant} then makes 
it possible to describe the kernel of the quotient map 
\[\Gamma^d(\pi): 
\Gamma^d M \onto \Gamma^d(M/N)\] 
induced by projection $\pi:M\onto M/N$.  
More generally, we note the following. 

\begin{lemma}\label{lem:quotient}
Let $A$ be a $\k$-algebra.  Suppose  
$N\subset M$ is an inclusion of $A$-modules  such that 
$M=N\oplus N'$ for some $\k$-submodule $N'\subset M$.
Then there is an exact sequence 
\[0\, \to\,
\< N \tensor M^{\tensor d-1} \>^{\Si_d}\, \longrightarrow \,
\Gamma^d M\, \xrightarrow{\Gamma^d(\pi)}\, 
\Gamma^d (M/N)\, \to\, 0\]
of $\Gamma^d A$-module homomorphisms.
\end{lemma}

\begin{proof}
It follows from (\ref{invariant}) that the required exact sequence 
of $\Gamma^d A$-modules is obtained 
by restriction 
from the exact sequence 
\[ 0\, \to\,  \< N\tensor M^{\tensor d-1} \> 
\, \longrightarrow \,  M^{\tensor d}\,
 \xrightarrow{\pi^{\tensor d}}\, 
(M/N)^{\tensor d}\, \to\, 0 \] 
of $A^{\tensor d}$-module homomorphisms. 
\end{proof}

We introduce some additional notation. 
Suppose $N_1, \dots, N_r\subset M$ 
is a finite collection 
of $\k$-submodules 
of some $M\in \P_\k$, 
and let $\mu \in \L_r(d)$.  
Then we write  
\[ N_{\tensor \mu} := N_1^{\tensor \mu_1}\tensor \dots \tensor N_r^{\tensor \mu_r}\] 
to denote the corresponding $\k$-submodule of $M^{\tensor d}$ and use the notation 
\begin{equation}\label{invariant2}
N_\mu := \<N_{\tensor \mu}\>^{\Si_d}  \subset \G^d M
\end{equation} 
for the $\k$-submodule of $\Si_d$-invariants. 
\smallskip

\section{Wreath Products and Generalized Schur-Weyl Duality}

Let us briefly recall 
the generalized Schur-Weyl duality \cite{EK1} 
which establishes a relationship between  
a wreath product algebra $A\wr \Si_d$ 
and a corresponding $A$-Schur algebra 
via their respective 
actions on a common tensor space.

\subsection{Wreath products}
Fix a $\k$-algebra $A$. 
The {\em wreath product algebra} $A\wr \Si_d$ 
is the $\k$-module $A^{\tensor d} \tensor \k \Si_d$, 
with multiplication 
defined by 
\begin{equation}\label{wr}
(x \tensor \rho ) \cdot (y \tensor \sigma)
:= x (y\rho^{-1}) \tensor \rho \sigma
\end{equation}
for all $x,y \in A^{\tensor d}$ and $\rho, \sigma \in \Si_d$. 
If $G$ is a finite group, then note for example that 
$(\k G) \wr \Si_d$ is isomorphic to 
the group algebra of the classical wreath product, 
$G\wr \Si_d := G^{d} \rtimes \Si_d$. 
\smallskip

Assume for the rest of the section that $A$ is free  
as a $\k$-module. 
We then identify the tensor power $A^{\tensor d}$ 
and group algebra  $\k\Si_d$  
as subalgebras of $A\wr \Si_d$ by setting 
\[A^{\tensor d}=A^{\tensor d} \tensor 1_{\Si_d}, \quad 
\k\Si_d=1_{A^{\tensor d}} \tensor \k \Si_d\] 
respectively.

\subsection{Generalized Schur-Weyl duality}

Suppose $n, d\in \N$. 
Write $\Vn := \k^n$ 
to denote the standard left $\Mat_n(\k)$-module, 
with basis elements 
\[ v_i:=  (0, \dots,  1, \dots, 0)\] 
for $i\in \set{1,n}$, considered as column vectors. 
Then for simplicity, let us write 
\[{\VV} :=  \Vn(A) = 
\k^n \tensor A\]
to denote the corresponding left $\Mat_n(A)$-module.

\smallskip

We may identify $\VV$ and $A^n$ as right $A$-modules, 
and it follows that the {\em tensor space}, 
$\VV^{\tensor d}$, 
is naturally a right 
$A^{\tensor d}$-module. 
A right action of 
$A\wr \Si_d$ on $\VV^{\tensor d}$
is then defined by setting 
\begin{equation}\label{space2}
w(x\cdot\sigma) := (wx)\sigma, \quad \text{for}\, \   
w\in \VV^{\tensor d}, \, 
x\in A^{\tensor d}\text{, and}\  
\sigma \in \Si_d.
\end{equation} 
More explicitly, suppose  $w=w_1\otimes\dots \otimes w_d$ 
and $x=x_1\otimes \dots\otimes x_d$, 
for some $w_i\in \VV$ and  $x_i\in A$. 
Then notice that 
\begin{equation*}
(wx)\sigma 
= (w_{1\sigma}x_{1\sigma}) \otimes \dots \otimes (w_{d\sigma}x_{d\sigma})
=(w\sigma)(x\sigma)
\end{equation*}
for any $\sigma \in \Si_d$. 
Hence, by \eqref{wr} we have 
\[w(\sigma \cdot x) = w((x\sigma^{-1})\cdot \sigma) 
= ( w(x\sigma^{-1}))\sigma
= (w\sigma)x.\]
It follows that \eqref{space2} is well-defined.

\begin{lemma}[{\cite[Lemma 5.7]{EK1}}]
The embedding 
$S^A(n,d) \hookrightarrow 
\Mat_n(A)^{\tensor d} \cong 
\End_{A^{\tensor d}}(\VV^{\tensor d})$
defines an algebra isomorphism 
\[S^A(n,d) \cong 
\End_{A\wr\Si_d}(\VV^{\tensor d})\]
for all $n,d \in \N$. 
\end{lemma}

Given $n\geq d$, 
let $\omega \in \Lambda_d(n)$ denote the weight 
$\omega = (1^d)= (1, \dots, 1, 0, \dots, 0)$.  
Then considering 
$\VV$ again as a left $\Mat_n(A)$-module, 
notice that  $\VV^{\tensor d}$ 
is equal to the 
left $S^A(n,d)$-module $\Gamma^{\omega} \VV$. 
\smallskip

For each weight $\mu\in \Lambda_d(n)$, 
define a corresponding element 
\[v_{\tensor \mu} := v_1^{\tensor \mu_1} \tensor \dots \tensor v_n^{\tensor \mu_n}\]
in the tensor space $\VV^{\tensor d}$.
\smallskip

The next result summarizes  (5.15) and (5.17) of \cite{EK1}. 

\begin{proposition}[{\cite{EK1}}]\label{prop:EK}
Assume that $n\geq d$. 
\begin{itemize}

\item[(i)]

There is a unique 
$(S^A(n,d), A\wr \Si_d)$-bimodule
isomorphism 
$S^A(n,d)\xi_\omega  \xrightarrow{\sim} 
\VV^{\tensor d}$ 
which maps $\xi_\omega \mapsto v_{\tensor \omega}$.

\item[(ii)]
There is an algebra isomorphism,
$A\wr \Si_d \xrightarrow{\sim}
\xi_\omega S^A(n,d) \xi_\omega$, 
given by: 
\[(x_1\tensor \dots \tensor x_d) \tensor \sigma 
\mapsto 
\xi_{1,1\sigma}^{x_1} \ast \dots \ast \xi_{d,d\sigma}^{x_d}.\]

\item[(iii)]
$\End_{S^A(n,d)}(\VV^{\tensor d})\, \cong\, A\wr \Si_d$. 
\end{itemize}
\end{proposition}

\section{Cauchy Decompositions} \label{S:Cauchy}

The Cauchy decomposition for symmetric algebras 
via Schur modules 
\cite{ABW} 
is an analogue of Cauchy's formula for 
symmetric functions \cite{Cauchy, Mac}.  
A corresponding 
decomposition for divided powers \cite{HK, Krause1}
is defined in terms of 
Weyl (or co-Schur)  modules.  
In this section, we describe  
a generalized Cauchy decomposition (Theorem \ref{thm:gen_Cauchy})
for divided powers   
of an $(A,B)$-bimodule  
with respect to a given filtration on the bimodule.

\subsection{Weyl modules}\label{ss:Weyl}

Weyl modules are defined in 
\cite[Definition II.1.4]{ABW} 
as the image of a single map from a tensor 
product of divided powers of a  module into a tensor 
product of exterior powers. 
We use an equivalent definition from the proof of \cite[Theorem II.3.16]{ABW}) 
which involves quotients of divided powers.
\smallskip

Throughout the section,  we fix some $d\in \N$. 
Suppose $\lambda \in  \L_d(\N)$,  
and let $M\in \P_\k$. 
For each pair $(i,t)$ with $1\leq i < l(\lambda)$ 
and $1\leq t \leq \lambda_{i+1}$, 
let us write 
\begin{equation}\label{eq:pm}
\lambda(i,t) =\, 
(\lambda_1,\, \dots,\, \lambda_{i-1} ,\, 
\lambda_i +t,\, \lambda_{i+1}-t,\, \lambda_{i+1},\, \dots, \lambda_m)\, 
\in\,  \L_d(\N).
\end{equation} 
Then write
$\gamma_{\lambda(i,t)} :  \Gamma^{\lambda(i,t)} M \to \Gamma^\lambda M$
to denote the standard homomorphism 
corresponding to the matrix 
\[ \gamma_{\lambda(i,t)} := 
 \rm{diag}(\lambda_1, \lambda_2, \dots ) 
+ t E_{i+1, i} - t E_{i+1, i}. \]
Similarly, let 
$\gamma^{\tr}_{\lambda(i,t)} :  
\Gamma^\lambda M \to 
\Gamma^{\lambda(i,t)} M$ 
denote the map 
corresponding to the transpose 
of the above matrix.

\begin{definition}[\cite{ABW}]
Suppose $M\in \P_\k$ 
and $\lambda\in  \Lambda_d^+(\N)$.  
Let $\square_\lambda(M)$ denote the 
$\k$-submodule of $\Gamma^\lambda M$ 
defined by 
\[\square_\lambda(M) :=\,  
\sum_{i \geq 1} \sum_{t=1}^{\lambda_{t+1} }
\rm{Im}(\gamma_{\lambda(i,t)} )
\subset \Gamma^\lambda M.\]
The {\em Weyl module}, 
$W_\lambda(M)$, is defined as 
the quotient $\k$-module 
\[ W_\lambda(M) 
:= \Gamma^\lambda M
\big / \square_\lambda(M). \]
\end{definition}

Let $A$ be a $\k$-algebra and
suppose now that $M \in A\lmod$.
Then $\square_\lambda(M)$ is a 
$\Gamma^d A$-submodule of $\Gamma^\lambda M$, 
since the standard homomorphisms are 
$\Gamma^d A$-module maps. 
It follows that $W_\lambda(M)$ is a $\Gamma^d A$-module.
In particular, 
$W_\lambda(\k^n)$ is an 
$S(n,d)$-module.

\subsection{The standard basis}
Consider a fixed partition 
$\lambda
=(\lambda_1, \lambda_2, \dots ) 
\in \Lambda^+_d(\N)$.
The {\em Young diagram} 
of $\lambda$ is the following subset of 
$\mathbb{N}\times \mathbb{N}$:
\[\set{\lambda} := \{(i,j)\ |\ 1\leq i\leq l(\lambda),\ 
1\leq j \leq \lambda_i\}.\] 
Suppose $\B$ is a finite totally ordered set. 
Let ${\Tab}_{\lambda}(\B)$ denote the set of all functions 
${\tt T}: \set{\lambda} \to \B$, 
 called {\em tableaux} ({\em of shape $\lambda$}). 
\smallskip

A tableau ${\tt T}$ will be identified 
with the diagram 
obtained by placing each value 
$\rm{T}_{i,j}:= {\tt T}(i,j)$ in the $(i,j)$-th entry of 
$\set{\lambda}$. 
For example if  ${\tt T}\in \Tab_{(3,2)}(\B)$, 
then we write 
\begin{equation}\label{eq:tab}
{\tt T}\, =
\begin{array}{ccc}
\rm{T}_{1,1} & \rm{T}_{1,2}  & \rm{T}_{1,3} \\[3pt]
\rm{T}_{2,1} & \rm{T}_{2,2}  &  
\end{array}
\end{equation}
We say that a tableau ${\tt T}$ 
is {\em row} ({\em column}) {\em standard} 
if each row (column)  is 
a nondecreasing (increasing)  function of $i$ (resp.~$j$), 
and ${\tt T}$ is {\em standard} 
if it is both row and column standard. 
\smallskip

Let $\st_\lambda(\B) \subset \Tab_\lambda(\B)$ 
denote the subset of all standard tableaux. 
This subset is nonempty if and only if 
$l(\lambda) \leq \sharp\B$. 
In particular, suppose $l(\lambda) \leq \sharp\B$ 
and assume the elements of $\B$ are listed as in \eqref{setB}. 
Then we write 
${\tt T}^\lambda = {\tt T}^\lambda(\B)$ 
to denote the standard tableau  in $\st_\lambda(\B)$ 
with entries 
 $ \rm{T}^{\lambda}_{i,j} := b_i^\B$ for all $(i,j) \in \set{\lambda}$. 
For example, 
if $d=7$, $\lambda = (4,2,1)$ and $\B=\set{1,3}$, then 
\begin{equation}\label{eq:tab}
{\tt T}^\lambda = 
\begin{array}{cccc}
1 & 1  & 1 & 1  \\
2  & 2  &   &  \\
3 & & & 
\end{array}
\end{equation}

Fix a free $\k$-module $V$ 
with finite ordered basis $\{x_b\}_{b\in \B}$. 
If ${\tt T}\in \Tab_{\lambda}(\B)$, 
then for $q=l(\lambda)$ and $i\in \set{1,q}$ we write 
\[{\tt T}_i:= {\tt T}(i,-) \in \seq^{\lambda_i}(\B)\]
to denote the  to the $i$-th row of ${\tt T}$, 
and we set 
\[x_{\tt T} :=\,  
x_{ {\tt T}_{1}} \tensor \cdots \tensor x_{  {\tt T}_{q} } \ 
\in \Gamma^\lambda V.\]
Notice that the set of 
$x_{\tt T}$ paramaterized by all row standard 
${\tt T}\in \Tab_\lambda(\B)$ 
forms a basis of $\Gamma^\lambda V$. 
\smallskip

The following result describes a basis for Weyl modules. 

\begin{proposition}[\cite{ABW}, Theorem III.3.16] \label{prop:standard}
Let $\lambda\in \Lambda^+(\N)$ 
and suppose $V$ is a free $\k$-module 
with a finite ordered basis $\{x_b\}_{b\in \B}$.  
Then the Weyl module $W_\lambda(V)$ is also a free $\k$-module, 
with basis given by the set of images 
\[\{\bar{x}_{\tt T} := \pi(x_{\tt T})\, |\ 
{\tt T}\in \st_\lambda(\B)\}\]
under the canonical projection 
$\pi: \Gamma^\lambda V \onto
\Gamma^\lambda V\big/\square_\lambda(V)$. 
\end{proposition}

This result shows for example that 
the Weyl module $W_\lambda(V)$ is nonzero 
if and only if  
$l(\lambda)\leq\sharp \B$. 
Another consequence of the proposition is that 
 $W_\lambda(M)$ is a projective $\k$-module for any  
$M\in \P_\k$ (cf.\,\cite[p.\,1013]{Krause1}).

\subsection{The Cauchy decomposition}

Suppose $M,N\in \P_\k$.  
The maps $\psi^d$  
appearing in (\ref{commute}) can be 
generalized as follows. 
If $\lambda\in \L^+_d(\N)$, let 
\[\psi^\lambda(M,N) :\, \Gamma^\lambda M \tensor \Gamma^\lambda N
\, \to\, \Gamma^d (M\tensor N)\]
denote the map defined via the composition  
\begin{align*}
\Gamma^\lambda M \tensor \Gamma^\lambda N
\isoto
( \Gamma^{\lambda_1} M \tensor \Gamma^{\lambda_1} N) &
\tensor \dots \tensor
(\Gamma^{\lambda_m} M \tensor \Gamma^{\lambda_m} N)\\[.1cm]
\xrightarrow{\, \psi \tensor \dots \tensor \psi }
 \Gamma^{\lambda_1} (M \tensor N) &
\tensor \dots \tensor
\Gamma^{\lambda_m} (M \tensor N)
  \xrightarrow{\nabla} 
\Gamma^d (M\tensor N),
\end{align*}
where the first map permutes tensor factors and the 
last map is multiplication in the bialgebra $\Gamma(M\tensor N)$.
\smallskip

Let us write $\Gamma^\lambda: \P_\k \to \P_\k$ to denote the
tensor product of functors 
\[ \Gamma^\lambda := 
\Gamma^{\lambda_1}\tensor \cdots \tensor \Gamma^{\lambda_m} \]
defined in the same way as \eqref{tensor}.  
Then it follows from Lemma \ref{psi} that 
the maps $\psi^\lambda(M,N)$ induce a natural transformation 
\begin{equation}
\psi^\lambda : \Gamma^\lambda \boxtimes \Gamma^\lambda \to 
\Gamma^\lambda(- \tensor - )
\end{equation}
of bifunctors $\P_\k\times \P_\k \to \P_\k$. 
\smallskip

The following lemma is a special case of \cite[Proposition III.2.6]{HK} 
which describes the relationship between   
 $\psi$-maps and standard homomorphisms. 

\begin{lemma}[\cite{HK}]\label{lem:HK}
Suppose
 $\lambda\in  \L_d^+(\N)$, and set $q=l(\lambda)$.
Given a pair   $U,V$ of free $\k$-modules of finite rank, 
the following diagram is commutative
\begin{equation*}
\begin{tikzcd}[row sep=large]
\Gamma^{\lambda(i,t)}U \tensor \Gamma^{\lambda}V
\ar[rr, "\rm{id}\tensor {\gamma}^\tr_{\lambda(i,t)}"]
\ar[d, "\, \gamma_{\lambda(i,t)}\tensor\rm{id}", shift left=1.7ex ] 
& & \Gamma^{\lambda(i,t)}U \tensor \Gamma^{\lambda(i,t)}V
\ar[d, "\, \psi^{\lambda(i,t)}" ]
& & 
\Gamma^{\lambda}U \tensor \Gamma^{\lambda(i,t)}V
\ar[ll, "{\gamma}^\tr_{\lambda(i,t)}\tensor  \rm{id}" ' ]
\ar[d, "\, \rm{id} \tensor \gamma_{\lambda(i,t)}" , shift right=1.7ex]
\\
\hspace{0.5cm}
\Gamma^{\lambda}U \tensor \Gamma^{\lambda}V
\ar[rr, "\psi^\lambda", shorten >=1.1em ]
& &
\hspace{-0.4cm}
\Gamma^d(U \tensor V)
& & 
\Gamma^{\lambda}U \tensor \Gamma^{\lambda} V
\hspace{0.5cm}
\ar[ll,"\psi^\lambda" ']
\end{tikzcd}
\end{equation*}
for any 
$i\in \set{1,q-1}$ and $t \in \set{1,\lambda_{i+1}}$.
\end{lemma}

Recalling the total order $\preceq$ on $\Lambda^+_d(\N)$ from Definition \ref{lex1}, 
write $\lambda^+$ to denote 
the immediate successor  
of a partition $\lambda$ 
and set $(d)^+ := \infty$.
The {\em Cauchy filtration}  is then defined as the chain 
\begin{equation*}
0\, =\, \F_{\infty}\, 
\subset  \F_{(d)}\, \subset\, 
\dots 
\subset\, \F_{(1,\dots, 1)}\, = \, \Gamma^d(M\tensor N)
\end{equation*}
where $\F_\lambda 
:= \sum_{\mu \succeq \lambda} \rm{Im}(\psi^\lambda)$.
\smallskip

The following result 
describes the factors of this filtration.

\begin{theorem} [{\cite[Theorem III.2.7]{HK}}] \label{thm:HK}
Let $U,V$ be free $\k$-modules of finite rank. 
Then for each $\lambda \in \L^+_d(\N)$, 
the map $\psi^\lambda$ induces an isomorphism 
\[\bar{\psi}^\lambda: 
W_\lambda(U) \tensor W_\lambda(V)
\, \isoto\,  \F_\lambda/ \F_{\lambda^+}\]
which makes the following diagram commutative: 
\begin{equation*}
\begin{tikzcd}[row sep=normal]
\Gamma^{\lambda}U \tensor \Gamma^{\lambda}V
\arrow[r,  shorten <=3, shorten >=3, " \psi^\lambda"]  
\arrow[d, two heads] 
& \F_{\lambda} 
 \arrow[d, two heads] \\
W_\lambda(U) \tensor W_\lambda(V)
 \arrow[r, xshift =5, shorten <=-3, shorten >=-3,  "\, \bar{\psi}^\lambda" ]
&\hspace{1em} \F_\lambda/ \F_{\lambda^+} 
\end{tikzcd}
\end{equation*}
Hence, the associated graded module of the Cauchy filtration is 
\[ \bigoplus_{\lambda\in \L^+_d(\N)} 
W_\lambda(U) \tensor W_\lambda(V).\]
\end{theorem}
 
\begin{proof} 
We recall the proof from \cite{HK}. 
It follows by definition that
$W_\lambda(U) \tensor W_\lambda(V)$ 
is the quotient of 
$\Gamma^\lambda U \tensor \Gamma^\lambda V$ 
by the submodule  
$\square_{\lambda}(U) \tensor \Gamma^\lambda V
+ \Gamma^\lambda U \tensor \square_{\lambda}(V) $. 
Hence, by Lemma \ref{lem:HK} we have 
\[\square_{\lambda}(U) \tensor \Gamma^\lambda V\, 
+\, \Gamma^\lambda U \tensor \square_{\lambda}(V) \
\subset\ \rm{Im}(\psi^{\lambda(i,t)})\ 
\subset\ \F_{\lambda^+},\]
since $\lambda(i,t) > \lambda$. 
This proves the existence of the induced map 
$\bar{\psi}^\lambda$ satisfying the given commutative square. 
It is clear that  
$\bar{\psi}^\lambda$ is surjective.   
Comparing the  ranks of 
$\Gamma^d(U\tensor V)$ 
and $\bigoplus_{\lambda\in \L^+_d(\N)} 
W_\lambda(U) \tensor W_\lambda(V)$ 
shows that $\bar{\psi}^\lambda$ must be an isomorphism 
for each $\lambda$. 
\end{proof}

Given free $\k$-modules  $U,V\in \P_\k$ with finite ordered bases 
 $\{x_b\}_{b\in \B}$ and  $\{y_c\}_{c\in \C}$, respectively, 
let $\F^{\,\prime}_{\hm{3}\lambda} 
\subset \Gamma^d(U\tensor V)$ 
denote the $\k$-submodule generated 
by 
\[\{ \psi^{\lambda}(x_S \tensor y_T)\ |\  
S \in \st_\lambda(\B),\,  
T \in \st_\lambda(\C)\}\]
where $\F^{\,\prime}_{\hm{3}\lambda}$ 
is nonzero only if  
$l(\lambda) \leq \rm{min}(\sharp \B  , \sharp \C)$.

\begin{corollary}\label{cor:HK}
For each $\lambda \in \L_d^+(\N)$,
the $\k$-submodule 
$\F^{\,\prime}_{\hm{3}\lambda} \subset \Gamma^d(U\tensor V)$ is free,  
and there is a corresponding decomposition: 
\[\Gamma^d(U\tensor V)
=\bigoplus_{\lambda} 
\F^{\, \prime}_{\hm{3}\lambda},
\quad \text{such that} \quad
\F_{\lambda}
= \bigoplus_{\mu\geq \lambda} \F^{\,\prime}_{\hspace{-.01cm}\mu} 
\quad \text{for all } \lambda \in \L^+_d(\N).\] 
\end{corollary}

\begin{proof} 
Suppose $\lambda \in \L^+_d(\N)$, and 
set $\T= \st_\lambda(\B)\times \st_\lambda(\C)$. 
By Proposition \ref{prop:standard},  
$\{\bar{x}_S \tensor \bar{y}_T
\, |\,  (S,T) \in  \T\}$ 
forms a basis of $W_\lambda(U) \tensor W_\lambda(V)$. 
So 
\[ \{ \bar{\psi}^\lambda(x_S \tensor y_T)\, |\, (S,T) \in \T\} \] 
gives a basis for $\F_\lambda/ \F_{\lambda^+}$ 
by Theorem \ref{thm:HK}. This shows that the subset 
\[\{ \psi^\lambda(x_S \tensor y_T) \, |\, (S,T) \in 
\T\}\ \subset\ \Gamma^d ( U \tensor V)\]
is linearly independent. 
Thus $\F^{\,\prime}_{\hm{3}\lambda}$ 
is a free $\k$-submodule.   
It is also clear that 
$\F_\lambda = 
\F_{\lambda^+} \oplus \F^{\,\prime}_{\hm{3}\lambda}$,
and the required decompositions follow by induction. 
\end{proof}

\subsection{Bimodule filtrations}
In the remainder of this section, 
we fix a set  $\{ J_1', \dots, J_r'\}$ 
of nonzero free $\k$-submodules, 
$J_i' \subset J$, such that setting 
\begin{equation}\label{eq:decomp}
J_j := \bigoplus_{1\leq i \leq j} J_i'\quad  
\text{for}\ \
 j \in \set{1, r}
\end{equation}
yields a chain 
\[ 0= J_0 \subset J_1 \subset \dots \subset J_r = J \] 
of $(A, B)$-bimodules. 
\smallskip

Recalling the notation \eqref{invariant2}, 
we then have for each $\mu\in \L_d(r)$  the following $\k$-submodules of $\G^d J$: 
\[ J'_\mu = \<J'_{\tensor \mu}\>^{\Si_d},  \qquad
J_\mu = \<J_{\tensor \mu}\>^{\Si_d}.\] 
Note first that $J_\mu$ is a $\Gamma^d(A\tensor B)$-submodule 
of $\Gamma^d J$, 
and hence a ($\G^d A, \G^d B$)-bimodule.
It is also not difficult to check that there is a decomposition of $J^{\tensor d}$ 
into free $\k$-submodules 
\[ J^{\tensor d} = 
\bigoplus_{\mu \in \L_d(r)} 
\< J'_{\tensor \mu}\>. \]
By taking $\Si_d$-invariants on both sides, we thus obtain the following   
decomposition 
\begin{equation}\label{eq:decomp2}
\Gamma^d J\ =\, \bigoplus_{\mu \in \L_d(r)} 
\<J'_{\tensor\mu}\> \cap \Gamma^d J\
=\, 
\bigoplus_{\mu \in \L_d(r)} 
J'_\mu.
\end{equation}

Next recall that the {\em dominance order} on 
$\L_d(r)$ is the partial order 
defined by setting 
$\mu \trianglelefteq \nu$ if 
\[ {\sum_{i\leq j} \mu_i }\  \leq \ 
{\sum_{i\leq j} \nu_i } \ 
\text{ for } \ j\in \set{1, r}. \] 
Notice that 
$J_\mu \subset J_\nu$ if and only if 
$\mu \trianglerighteq \nu$.  
We further have 
$J_\nu = J'_\nu \oplus 
\sum_{\mu \triangleright \nu} 
J_\mu$,
and it follows by induction that 
\begin{equation}\label{eq:decomp_nu}
J_{\nu}
= \bigoplus_{\mu \trianglerighteq \nu} 
J'_{\mu}
\end{equation}
for all $\nu \in \L_d(r)$, which generalizes  
the decomposition \eqref{eq:decomp2} of $\Gamma^d J$.
\smallskip

Consider the map 
$\nabla: \Gamma^\mu J \to \Gamma^d J$ 
given by $r$-fold (outer) 
multiplication in 
$\Gamma( J)$, for some 
 $\mu \in \L_d(r)$. 
Note that the restriction 
\[\nabla^\mu : 
\Gamma^{(\mu)}(J_1, \dots, J_r) \to 
\Gamma^d J \quad 
\left(\text{resp.~} '\nabla^\mu : 
\Gamma^{(\mu)}(J'_1, \dots, J'_r) \to  
\Gamma^d J\right)\] 
is a $(\Gamma^d A, \Gamma^d B)$-bimodule 
(resp.~$\k$-module) homomorphism. 

\begin{lemma}\label{lem:image}
Suppose $\nu \in \L_d(r)$. Then
\\[-.25cm]
\begin{enumerate}
\item
 $J'_\nu =  \rm{Im}\, '\nabla^\nu$,
\\[-.15cm]
\item
$'\nabla^{\nu} : \Gamma^{(\nu)}(J'_1,\dots, J'_r) \isoto 
 J'_\nu$ is an isomorphism of
$\k$-modules, 
\\[-.15cm]
\item 
$J_\nu = \sum_{\mu\trianglerighteq\nu}
\rm{Im}\, \nabla^\mu$, 
summing over $\mu\in \L_d(r)$. 
\end{enumerate}
\end{lemma}

\begin{proof}
For each $\mu \in \L_d(r)$, 
write 
$M_\mu$, $M'_\mu$ 
to denote the images of 
$\Gamma^{(\mu)}(J_1, \dots, J_r)$ and  
$\Gamma^{(\mu)}(J'_1, \dots, J'_r)$,
respectively, 
under the map   
$\nabla^\mu: \Gamma^\mu J \to \Gamma^d J$.
It is then clear from the definitions that 
$M'_\mu \subset J'_\mu$
and similarly 
$M_\mu = 
\rm{Im}\,\nabla^{\mu} 
 \subset J_\mu$, 
for all $\mu$.

It  follows inductively from  the isomorphism 
(\ref{eq:expon2})  that there is a decomposition 
\begin{equation*}
\Gamma^d J\ =\ 
\Gamma^d(J'_1\oplus \dots \oplus J'_r) \ = 
\bigoplus_{\mu\in \L_d(r)} M'_\mu\ \, \subset 
\bigoplus_{\mu\in \L_d(r)} J'_\mu 
\end{equation*}
It thus follows from 
(\ref{eq:decomp2}) 
that $J'_\mu =M'_\mu \cong 
\Gamma^{(\mu)}(J'_1, \dots, J'_r)$ 
which shows (1) and (2). 
Since $J_{\mu} \subset J_{\nu}$ whenever 
$\mu \trianglerighteq \nu$, 
it follows from (\ref{eq:decomp_nu}) that 
\[J_{\nu}\, =\, 
\bigoplus_{\mu\trianglerighteq \nu} M'_\mu \, 
\subset \, 
\sum_{\mu\trianglerighteq \nu} M_\mu 
\, \subset \,  
\sum_{\mu\trianglerighteq \nu} J_\mu 
\, \subset \, J_{ \nu}\] 
showing (3). 
\end{proof}

Recall the lexicographic ordering $\leq$ 
on $\L_d(r)$ from Definition \ref{lex1}, and notice that 
there is a chain 
of $(\Gamma^d A,  \Gamma^d B)$-sub-bimodules 
\[ 0 \, \subset \, \Gamma^d (J_1) \, = \, 
J_{\geq (d,0,\dots, 0)} \, \subset \, 
\cdots \, \subset \,  
J_{\geq \nu} \, \subset \, \cdots \, \subset \,  
J_{\geq (0, \dots, 0, d)} \, = \, \Gamma^d J \]
where 
$J_{\geq \nu} := \sum_{\mu\geq \nu} J_\mu$ 
for each $\nu \in \L_d(r)$. 
Since the lexicographic ordering refines the 
dominance order, it follows from 
(\ref{eq:decomp_nu}) that 
\begin{equation}\label{eq:sum_nu}
J_{\geq\nu} = J_{> \nu} \oplus J'_\nu
\end{equation} 
for all $\nu$. Thus  
\[J_{\geq \nu} = \sum_{\mu \geq \nu} 
\rm{Im}(\nabla^{\mu})\] 
by the preceding lemma.
This  allows us to describe the quotients 
$J_{\geq \nu}/ J_{>\nu}$ 
as follows. 

\begin{proposition}\label{prop:commute}
Let $\nu \in \L_d(r)$.  
Then $\nabla^\nu$ induces an isomorphism 
\[\bar{\nabla}^\nu: 
\Gamma^{(\nu)}(J_1/J_0, \dots, J_r/J_{r-1})
\cong  
J_{\,\geq\nu} / J_{\,>\nu}\]
which yields a commutative square 
of $(\Gamma^d A, \Gamma^d B)$-bimodule homomorphisms  
\begin{equation*}
\begin{tikzcd}[row sep=large]
\Gamma^{(\nu)}(J_1, \dots, J_r)
\arrow[r, " \nabla^\nu"]  
\arrow[d, two heads, "\pi_{\phantom{1}}^\nu"' ] 
& J_{\,\geq\nu}
 \arrow[d, two heads,  "\pi"  ] \\
\Gamma^{(\nu)}(J_1/J_{0}, \dots, J_r/J_{r\-1})
 \arrow[r, " \bar{\nabla}^\nu"]
& J_{\,\geq}(\nu) / J_{\,>}(\nu) 
\end{tikzcd}
\end{equation*}
where $\pi^\nu$ denotes 
the tensor product of functorial maps  
$\Gamma^{\nu_j}(\pi_j)$ 
associated to the projections, 
$\pi_j: J_j \to J_j/J_{j-1}$, 
for $j=1, \dots, r$, 
and where $\pi$ is also projection. 
\end{proposition}

\begin{proof}
We first verify that 
$\ker \pi^\nu \subset J_{>\nu}$  
in order to show the existence of the map 
$\bar{\nabla}^\nu$ satisfying the above diagram. 
If $1\leq j \leq r$, consider the 
$(\Gamma^d A , \Gamma^d B)$-sub-bimodule 
\[K_j :=\, 
\Gamma^{(\nu^1)}(J_1,\dots, J_r) \tensor 
\ker \Gamma^{\nu_j}(\pi_j) 
\tensor 
\Gamma^{(\nu^2)}(J_1,\dots, J_r)\] 
where 
$\nu^1 = (\nu_1, \dots, \nu_{j-1}, 0, \dots, 0)$ 
and $\nu^2 = (0, \dots, 0, \nu_{j+1}, \dots, \nu_r)$. 
Then 
$\ker \pi^\nu =\sum_{j=1}^r K_j$, 
and we must show that $K_j \subset J_{>\nu}$ 
for all $j$. 
\smallskip

Now $K_j = 0$, if either  $j=1$ or $\nu_j=0$.  
If $K_j\neq 0$ and $1\leq t \leq \nu_j$, 
let $\nu(j,t)\in \L_d(r)$ be defined 
as in (\ref{eq:pm}). 
Since $\nu(j,1) > \nu$, it suffices to show that 
$\nabla^\nu(K_j) \subset 
\rm{Im}\, \nabla^{\nu(j,1)}$ for all such $j$.  
The fact that 
$\nu$ and $\nu(j,1)$ are equal except for 
entries in the $j$-th and  $(j-1)$-st positions 
allows us to simplify to the case $r=2$. 

So we may assume $\nu = (\nu_1, \nu_2)$.  
Then for $j=2$, we have 
 $\nu(2,1) = (\nu_1+1, \nu_2-1)$.
In this case $K_2 = \ker(\pi^\nu)$, 
and it follows by Lemma \ref{lem:quotient} that 
\[ K_2 =\, \Gamma^{\nu_1} J_1 \tensor 
J_{(1,\nu_2-1)} \subset \, 
\Gamma^{(\nu)}(J_1, J_2).\]
Notice by Lemma \ref{lem:image} 
that $J_{(1,\nu_2-1)}$ 
is equal to the image of the map 
\[\nabla^{(1,\nu_2-1)}: 
J_1 \tensor \Gamma^{\nu_2-1} (J_2)
\to \Gamma^{\nu_2}(J_2).\] 
By associativity of multiplication in $\Gamma(J)$, we also have a 
commutative diagram 
\begin{equation*}
\begin{tikzcd}
[column sep=small, row sep=small]
& \Gamma^{((\nu_1, 1, \nu_2\-1))} (J_{1} , J_{1} , J_2)
\arrow[ddl, "\rm{id}\tensor \nabla^{(1,\nu_2-1)}"', shorten >=-0.2em]  
\arrow[ddr, "\nabla \tensor \rm{id}"]  
\\
\\
\Gamma^{(\nu)}(J_1, J_2)\, 
 \arrow[dr, " \nabla^\nu"' , shorten <=-0.5em]  
& 
& 
\ \, \Gamma^{(\nu(2,1))}(J_1, J_2) 
\arrow[dl, " \nabla^{\nu(2,1)}" , shorten >=-0.2em, shorten <=-0.9em ]  \\
& 
J_{\geq\nu}
& 
\end{tikzcd}
\end{equation*}
It follows that 
$\nabla^\nu(K_2) \subset \rm{Im}\, \nabla^{\nu(2,1)}$, 
which shows the existence of 
$\bar{\nabla}^\nu$. 
To complete the proof, 
note that the restriction 
$\pi^\nu |_{\Gamma^{(\nu)}(J'_1, \dots, J'_r)}$ 
is a $\k$-module isomorphism. 
The map $(\pi\circ \nabla^\nu)
|_{\Gamma^{(\nu)}(J'_1, \dots, J'_r)}$ 
is  also a $\k$-module isomorphism by 
Lemma \ref{lem:image}. It follows that 
$\bar{\nabla}^\nu$ is an isomorphism 
by commutativity. 
\end{proof}

\subsection{Multitableaux}
Suppose $\{\B_j\}_{j \in\set{1,r}}$ is a collection of finite 
totally ordered sets, and let 
$\bs \lambda \in \L^{+}_d(\N)^r$ 
be an $r$-multipartition. Elements of the set 
\[ \Tab_{\bs \lambda}(\B_1, \dots, \B_r) :=
\Tab_{\lambda^{(1)}}(\B_1) \times \dots \times \Tab_{\lambda^{(r)}}(\B_r). \]
are called {\em multitableaux of shape $\bs \lambda$} 
(or $\bs \lambda$-{\em multitableaux}). 
\smallskip 

We say that a $\bs \lambda$-multitableau, 
$\mathbf T= ({\tt T}^{(1)}, \dots, {\tt T}^{(r)})$, 
is {\em standard} if 
each component ${\tt T}^{(j)}$ is a standard $\lambda^{(j)}$-tableau. 
The subset of standard ${\bs \lambda}$-multitableaux is denoted 
 \[ \st_{\bs \lambda}(\B_\ast) = \st_{\bs \lambda}(\B_1, \dots, \B_r).\]
If $(n_1, \dots, n_r) \in \N^r$
is the sequence of integers with $n_j := \sharp \B_j$ for all $j$, 
then it follows from \eqref{eq:tab} that 
$\st_{\bs \lambda}(\B_\ast)$ is non-empty 
if and only if 
$\bs \lambda$ belongs to the subset 
${\L}^+_d(n_1, \dots, n_r) \subset \L^+_d(\N)^r$. 
In this case, we write 
$\mathbf T^{\bs \lambda} = \mathbf T^{\bs \lambda}(\B_\ast)$ 
to denote the standard $\bs \lambda$-multitableau 
\begin{equation}\label{multi-tab}
\mathbf T^{\bs \lambda} 
:= ({\tt T}^{\lambda^{(1)}}, \dots,{\tt T}^{\lambda^{(r)}}). 
\end{equation}

Suppose $\nu = (\nu_1, \dots, \nu_r) \in \L_d(r)$.  
There is a corresponding $r$-multipartition 
$(\nu) := ((\nu_1), (\nu_2), \dots, (\nu_r)) \in \L^{+}_d(\N)^r$. 
For any $m\in \N$, let us write 
$(1^m) := (1,\dots, 1) \in \L^+_m(\N)$, 
and set $(1^0) = 0$. 
Then we also have an element 
\[ (\nu)' := ((1^{\nu_1}), 
(1^{\nu_2}), \dots, (1^{\nu_r}))
\in \L^{+}_\nu(\N). \]
Recalling the total order $\preceq$ from  
Definition \ref{lex2}, notice that 
$(\nu)' \preceq \bs \lambda \preceq (\nu)$ 
for all $\bs \lambda \in  \L_\nu^+(\N)$. 
We also write $\bs \lambda^+$ to 
denote the immediate successor of any  
$\bs \lambda \in \L^{+}_d(\N)^r$ 
and set $((d))^+ = \infty$. 

\subsection{Generalized Weyl modules}\label{ss:gen_Weyl}
Given $\bs \lambda \in \L^{+}_d(\N)^r$ and projective modules $M_j \in \P_\k$ for $j\in \set{1,r}$, 
we will use the notation 
\[\Gamma^{\bs \lambda}
(M_\ast) := \bigotimes_{j} 
\Gamma^{\lambda^{(j)}} M_j,
\qquad W_{\bs \lambda}
(M_\ast) := \bigotimes_{j} 
W_{\lambda^{(j)}} M_j\]
in what follows. 
 The outer tensor product $-\boxtimes-$\,, 
 defined in Section \ref{ss:functor},
yields 
 corresponding functors 
$\Gamma^{\bs \lambda}, W_{\bs \lambda}: \P_\k^{\times r} \to \P_\k$  
defined by 
\[\Gamma^{\bs \lambda} := 
\Gamma^{\lambda^{(1)}} \boxtimes \cdots \boxtimes \Gamma^{\lambda^{(r)} }
\quad \text{and} \quad  
W_{\bs \lambda} := 
W_{\lambda^{(1)}} \boxtimes \dots \boxtimes W_{\lambda^{(r)} }.\] 
Since Weyl modules are quotients of divided powers, 
it follows that there is a natural projection 
$\pi: \Gamma^{\bs \lambda} \onto W_{\bs \lambda}$. 
\smallskip

Suppose $V_1, \dots, V_r \in \P_\k$ are free $\k$-modules, 
and suppose 
$\{x^{(j)}_b\}_{b\in \B_j}$ 
is a finite ordered basis of $V_j$ for each $j\in \set{1,r}$.  
Given a  multitableau $\mathbf T
\in \Tab_{\bs \lambda}(\B_1, \dots, \B_r)$, 
there is a corresponding element 
\[ x_{\mathbf T}  := 
\bigotimes_{j}  x^{(j)}_{{\tt T}^{(j)}} 
 \in\,  \Gamma^{\bs \lambda}(V_\ast)\]
whose image in 
$W_{\bs \lambda}(V_\ast)$ 
is denoted $\bar{x}_{\bm T} := \pi(x_{\bm T})$. 
The next result follows easily from Proposition \ref{prop:standard}.

\begin{lemma}
Let  $\bs \lambda \in \L^{+}_{d}(\N)^r$ 
be an  $r$-multipartition, and let 
$V_1, \dots, V_r$ be  free $\k$-modules 
with bases as above. 
The  set of images 
$\{\bar{x}_{\bm T}\, |\ 
\mathbf T\in \st_{\bs \lambda}(\B_1, \dots, \B_r)
\}$ 
forms a basis of the 
free $\k$-module 
$W_{\bs\lambda}( V_\ast )$
parametrized by standard $\bs \lambda$-multitableaux. 
In particular, we have 
$W_{\bs \lambda}(V_\ast) = 0$ unless 
$\bs \lambda \in \L^+_d(\sharp \B_1, \dots, \sharp \B_r)$. 
\end{lemma}

Suppose $\nu \in \Lambda_d(r)$ 
and fix some projective modules $M_j, N_j \in \P_\k$ for $j\in \set{1,r}$.
Using notation similar to the above, we write 
\[ \Gamma^{(\nu)}(M_\ast\tensor N_\ast) := 
\bigotimes_{j} 
\Gamma^{\nu_j} (M_j\tensor N_j).\] 
Given $\bs \lambda \in \L^+_\nu(\N)$, 
we then define a map  
\[\psi^{\bs \lambda}: 
\Gamma^{\bs \lambda}(M_\ast) \tensor 
\Gamma^{\bs \lambda}(N_\ast)\to 
\Gamma^{(\nu)}(M_\ast\tensor N_\ast)\]
via the composition 
\begin{align}
\big\{\bigotimes_j 
\Gamma^{\lambda^{(j)}} M_j \big\}
&\tensor 
\big\{
\bigotimes_j  \Gamma^{\lambda^{(j)}} N_j \big\}  \label{eq:new_psi}\\ 
 \cong\  
&
\bigotimes_j \big\{
\Gamma^{\lambda^{(j)}}(M_j) \tensor 
\Gamma^{\lambda^{(j)}}(N_j)\big\} 
\xrightarrow{\, \psi \tensor\dots \tensor \psi\, }
\bigotimes_j 
\Gamma^{\nu_j}(M_j\tensor N_j).  \nonumber
\end{align}

Note that if  $M_j \in \Amod$ and $N_j \in \Bmod$ 
for all $j$, then $\psi^{\bs \lambda}$ 
is a homomorphism 
of $(\Gamma^d A, \Gamma^d B)$-bimodules
by Lemma \ref{psi}.1.

\subsection{Generalized Cauchy filtrations of bimodules}
Fix a chain $(J_j)_{j\in \set{0, r}}$ of $(A, B)$-bimodules. 
For each $j\in \set{1,r}$, suppose there exists an isomorphism
\begin{equation}\label{eq:alpha_isom}
\alpha_j: 
 J_j/ J_{j-1} \equi  U_j \tensor V_j 
\end{equation}
of $(A, B)$-bimodules  
for some $U_j\in \Amod$ and $V_j\in  \Bmod$. 
Assume for all $j$ that 
$U_j$ and $V_j$ 
are free as $\k$-modules,  
with finite ordered bases 
$\{x^{(j)}_b\}_{b\in \B_j}$ and 
$\{y^{(j)}_c\}_{c\in \C_j}$, respectively. 
Assume further that  $\{J'_j\}_{j\in [r]}$ is any collection of free $\k$-submodules of $J_r$ 
such that (\ref{eq:decomp}) holds.  
\smallskip

We first define a filtration of 
$\Gamma^{(\nu)}(U_{\ast}\tensor V_{\ast})$    
for some fixed weight $\nu \in \L_d(r)$. 
For each $r$-multipartition 
$\bs \lambda \in \L^{+}_\nu(\N)^r$, 
let us write 
\[\F_{\bs \lambda, (\nu)} := 
\sum_{\bs \lambda \leq \bs \mu \leq (\nu)} 
\F_{\mu^{(1)}}(U_1, V_1) \tensor \cdots \tensor \F_{\mu^{(r)}}(U_r, V_r)\]
which is a sum of sub-bimodules 
of $\Gamma^{(\nu)}(U_{\ast}\tensor V_{\ast})$. 
It follows that there is a chain 
of sub-bimodules: 
\begin{equation}\label{eq:chain}
0 =: \, 
\F_{(\nu)^+,(\nu)} \subset\,  
\F_{\bs (\nu), (\nu)} \subset
\dots
\subset\, 
\F_{(\nu)',(\nu)} =\, \Gamma^{(\nu)}(U_\ast \tensor V_\ast).
\end{equation}
\vspace{-9pt} 

Recalling  \eqref{eq:new_psi}, notice that 
for each  $\bs \lambda \in \L^{+}_\nu(\N)^r$ 
we have 
\[\F_{\bs \lambda, (\nu)} = 
\sum_{\bs \lambda \preceq \bs \mu \preceq (\nu)} 
\rm{Im}(\psi^{\bs \mu}).\] 
Note also that $\F_{\bs \lambda, (\nu)}$
contains the $\k$-submodule 
\[ \F'_{\bs \lambda} 
:= \bigotimes \F'_{\lambda^{(j)}}(U_j, V_j). \] 
It then follows by Corollary \ref{cor:HK} that 
$\F'_{\bs \lambda}$ is a free $\k$-submodule, 
with the set 
\begin{equation}\label{eq:basis}
\{ \psi^{\bs \lambda}(x_{\bm S} \tensor y_{\bm T}) \, |\ 
\mathbf S \in \st_{\bs \lambda}(\B_1, \dots, \B_r),\, 
\mathbf T \in \st_{\bs \lambda}
(\C_1, \dots, \C_r)
\}
\end{equation}
as a basis.

\begin{proposition}\label{prop:gen_HK}
Suppose $\nu \in \L_d(r)$.  
Then for each 
$\bs \lambda \in \L^+_\nu(\N)$,
the map 
\[\psi^{\bs \lambda}: 
\Gamma^{\bs \lambda}(U_\ast) \tensor \Gamma^{\bs \lambda}(V_\ast)
\to \F_{\bs \lambda, (\nu)}\]
induces an isomorphism 
\[
\bar{\psi}^{\bs \lambda}:\, 
\F_{\bs \lambda,(\nu)} / \F_{\bs \lambda^+,(\nu)} 
\, \isoto\, 
 W_{\bs \lambda} (U_\ast) \tensor W_{\bs \lambda} (V_\ast)
\]
of bimodules. We also have decompositions  
\begin{equation}\label{eq:decomp_HK}
\Gamma^{(\nu)}(U_\ast \tensor V_\ast) = 
\bigoplus_{\bs \lambda \in \L^+_\nu(\N)} \F'_{\bs \lambda}, \qquad 
\F_{\bs \lambda, (\nu)} = 
\bigoplus_{\bs \lambda \preceq \bs \mu \preceq (\nu)} \F'_{\bs \mu} 
\end{equation}
into free $\k$-submodules.  
\end{proposition}

We now wish to  lift the filtrations 
\eqref{eq:chain},
for varying $\nu$, 
to a single filtration of $\Gamma^d J$, 
with $J=J_r$ as above.   
First note that there is an isomorphism 
\begin{equation*}
\phi_\nu: J_{\geq \nu}/ J_{>\nu} \isoto \Gamma^{(\nu)}(U_\ast \tensor V_{\ast})
\end{equation*}
satisfying the following commutative triangle of
$(\Gamma^dA, \Gamma^d B)$-bimodule isomorphisms: 
\begin{equation}\label{eq:isom_phi}
\begin{tikzcd}[row sep=large, column sep=small]
\bigotimes \Gamma^{\nu_j}(J_j/J_{j-1})
 \arrow[rr, " \Gamma^{(\nu)}(\alpha_\ast)"]   \arrow[dr, " \bar{\nabla}^\nu"']
&  & 
\bigotimes \Gamma^{\nu_j}\big(U_j\tensor V_j \big)
\\
& J_{\,\geq\nu} / J_{\,>\nu}
\arrow[ur, " \phi_\nu"']
\end{tikzcd}
\end{equation}
where 
$\Gamma^{(\nu)}(\alpha_\ast) = \bigotimes \Gamma^{\nu_j}(\alpha_j)$ 
is a tensor product of isomorphisms induced by the maps (\ref{eq:alpha_isom}) 
and $\bar{\nabla}^\nu$ is defined in Proposition \ref{prop:commute}.
We then have a surjective map 
\[\hat{\phi}_\nu: J_{\geq \nu}\  \onto\ 
\Gamma^{(\nu)}(U_\ast \tensor V_\ast)\]
obtained by composing $\phi_\nu$ with the projection 
$\pi: J_{\geq \nu} \onto J_{\geq \nu}/ J_{>\nu}$.

\begin{definition} 
Suppose $\bs \lambda \in \L^+_d(\N)^r$ and set 
$\nu = |\bs \lambda|$. 
Then  define $\J_{\bs \lambda}$ to be the sub-bimodule 
of $J_{\geq \nu}$, 
corresponding to the inverse image of 
$\F_{\bs \lambda, (\nu)}$ 
under the map $\phi_\nu$ considered above. 
The {\em generalized Cauchy filtration} of $\Gamma^d J$ 
is then defined as the chain 
\begin{equation} \label{eq:gen_Cauchy}
0 = \J_{\infty} \subset \J_{((d))} \subset \dots
\subset 
\J_{\bs \lambda^+} \subset \J_{\bs \lambda} \subset \dots
\subset
\J_{((1^d))} = \Gamma^d J
\end{equation}
of $(\Gamma^d A, \Gamma^d B)$-bimodules 
parametrized by multipartitions 
$\bs \lambda \in \L^+_d(\N)^r$.
\end{definition}

We next define a decomposition of $\Gamma^d J$
via certain $\k$-submodules, 
$\J'_{\bs \lambda} \subset \J_{\bs \lambda}$. 
Recall from (\ref{eq:decomp}) 
 that 
 $J_j = J'_j \oplus J_{j-1}$, for all $j$.
 For each $j\in \set{1,r}$, let 
\[\alpha'_j: J'_j \xrightarrow{\sim} U_j\tensor V_j\] 
denote the isomorphism defined via the  composition 
\[\begin{tikzcd}[column sep=small]
J'_j \arrow[r,hook] & J_j \arrow[r, two heads] & J_j / J_{j-1}
\arrow[rr, "\alpha_j"] & &
U_j \tensor V_j .
\end{tikzcd}\]
Similar to (\ref{eq:isom_phi}),
there is a resulting $\k$-module isomorphism 
\[\phi'_\nu: J'_\nu\, \xrightarrow{\sim}\, \Gamma^{(\nu)}(U_\ast \tensor V_\ast) \]
satisfying the following commutative triangle of isomorphisms: 
\begin{equation}\label{eq:isom_phi'}
\begin{tikzcd}[row sep=large, column sep=small]
\Gamma^{(\nu)}(J'_\ast)
 \arrow[rr, " \Gamma^{(\nu)}(\alpha'_\ast)"]   
 \arrow[dr, " {}'\nabla^\nu"']
&  & 
\Gamma^{(\nu)}(U_\ast \tensor V_\ast )
\\
& 
J'_{\nu} 
\arrow[ur, " \phi'_\nu"']
\end{tikzcd}
\end{equation}
where 
\[\Gamma^{(\nu)}(\alpha'_\ast) := 
\bigotimes \Gamma^{\nu_j}(\alpha'_j)\] 
and where  
${}'\nabla^\nu$ is restriction of $r$-fold multiplication 
as in Lemma \ref{lem:image}.(i). 
We write  
\[\J'_{\bs \lambda}\, := \,
(\phi'_\nu)^{-1}(\F'_{\bs \lambda})\]
to denote the inverse image 
of $\F'_{\bs \lambda}$ under $\phi'_\nu$.

\begin{lemma} \label{lem:gen_decomp}
There exist  decompositions into free $\k$-submodules
\[\Gamma^d J = 
\bigoplus_{\bs \lambda \in \L^+_d(\N)^r } 
\J'_{\bs \lambda},
\qquad \text{and} \qquad 
\J_{\bs \lambda} = 
\bigoplus_{\bs \lambda \preceq \bs \mu \prec \infty} 
\J'_{\bs \mu}
\quad
\text{for each } \bs \lambda. 
\]
\end{lemma}
 
\begin{proof}
It follows by definition from 
 (\ref{eq:isom_phi}) and (\ref{eq:isom_phi'}) 
that $\phi'_\nu$ can be obtained from $\hat{\phi}$ by restriction.  
In particular, we have a commutative diagram: 
\begin{equation}\label{eq:gen_decomp}
\begin{tikzcd}[column sep = large, ]
J'_\nu \arrow[d, shorten <=1pt, shorten >=1, tail] 
\arrow[r, "\phi'_\nu"]
&\Gamma^{(\nu)}(U_\ast \tensor V_\ast)
\arrow[d, no head, shift right=7.2,   shorten >=1]
\arrow[d, no head, shift right=8, shorten >=1]
\\
J_{\geq \nu} 
 \arrow[r, two heads, "\hat{\phi}_\nu"]
&\Gamma^{(\nu)}(U_\ast \tensor V_\ast)
\end{tikzcd}
\end{equation}
Since $J_{\geq \nu} =  J_{>\nu}\oplus J'_\nu$ 
by (\ref{eq:sum_nu}), 
we further have a decomposition  
\begin{equation}\label{eq:sum_phi}
\hat{\phi}_\nu^{-1}(N) \, = \, 
J_{>\nu}\oplus (\phi'_\nu)^{-1}(N) 
\end{equation}
for any $\k$-submodule 
$N \subset \Gamma^{(\nu)}(U_\ast \tensor V_{\ast})$. 
If we set $N= \F_{\bs \lambda, (\nu)}$ in the above, 
then it follows from  (\ref{eq:decomp_HK}) that 
\begin{equation*}
\J_{\bs \lambda}  \, = \, 
J_{>\nu}\oplus
\bigoplus_{\bs \lambda \preceq \bs\mu \preceq (\nu)}
\J'_{\bs \lambda}  
\end{equation*}
for each $\bs \lambda \in  \L_\nu(\N)$.  
The decomposition of $\J_{\bs \lambda}$ 
now follows by induction 
since $J_{> \nu} = \J_{(\nu_+)}$, 
where $\nu_+$ denotes an immediate successor of $\nu$ in 
the lexicographic order on $\L_d(r)$. 
The decomposition for $\Gamma^d J = \J_{(1^d)}$ 
follows as a special case. 
\end{proof}

Now suppose 
$\bs \lambda \in \L^+_d(\N)^r$. 
To each element of the 
basis (\ref{eq:basis}), we associate a 
corresponding element in 
$\J'_{\bs \lambda}$, defined by
\begin{equation}\label{z_basis} 
z_{\bm S, \bm T} := \, 
\big({}'\nabla^\nu \circ \Gamma^{(\nu)}(\alpha'_\ast)^{-1} \circ \psi^{\bs \lambda}\big) 
(x_{\bm S} \tensor y_{\bm T}). 
\end{equation} 
Since the map appearing in \eqref{z_basis} is 
a composition of isomorphisms, 
it follows that the set 
\[
\{ z_{\bm S, \bm T}  \mid 
\bm S \in \st_{\bs \lambda}(\B_\ast),\, 
\bm T \in \st_{\bs \lambda}(\C_\ast) \}
\] 
forms a basis of $\J'_{\bs \lambda}$.
\smallskip

Let $(m_1,\dots, m_r) \in \N^r$ be the sequence 
defined by 
\[m_j := \min(\sharp \B_j, \sharp \C_j)\]
for all $j$, and set 
\[\bs \L: = \L_r^+(m_1, \dots, m_r).\] 

\begin{remark}\label{rmk:multi_tab}
Suppose 
$\bs \lambda \in \L^+(\N)^r$. 
If $\bs \lambda$ belongs to $\bs \L\subset \L^+(\N)^r$, 
then $\st(\B_\ast)$ and $\st(\C_\ast)$ 
are both non-empty since they contain 
the elements 
${\tt T}^{\bs \lambda} = {\tt T}^{\bs \lambda}(\B_\ast)$ 
and
${\tt T}^{\bs \lambda} = {\tt T}^{\bs \lambda}(\C_\ast)$ 
defined in \eqref{multi-tab}, respectively. 
We thus have  
$\J'_{\bs \lambda} \neq 0$ if and only if 
$\bs \lambda \in \bs \L$. 
\end{remark}

Let $\bs \lambda \in \bs \L$.  Since 
$\J_{\bs \lambda} = \J'_{\bs \lambda} \oplus \J_{\bs \lambda^+}$ 
by Lemma \ref{lem:gen_decomp}, 
it follows that 
$\J_{\bs \lambda} / \J_{\bs \lambda^+}$ 
is a free $\k$-module with basis
\[ \left\{\bar{z}_{\bm S, \bm T} \mid 
\bm S \in \st_{\bs \lambda}(\B_\ast),\,  
\bm T \in \st_{\bs \lambda}(\C_\ast) 
\right\} \]
where $\bar{x} :=
x + \J_{\bs \lambda^+}$
denotes the image 
of $x\in \J_{\bs \lambda}$
in the quotient.  

\begin{definition}
Given $\bs \lambda \in \bs \L$, 
define a pair of $\k$-submodules 
\[\U_{\bs \lambda},\  
\V_{\bs \lambda} \, \subset\ 
\J_{\bs \lambda} / \J_{\bs \lambda^+}\]
 generated by the subsets 
\[ \left\{\bar{z}_{\bm S, \bm T^{\bs \lambda}} \mid 
\bm S \in \st_{\bs \lambda}(\B_\ast) \right\} 
\quad  \text{and} \quad
\left\{\bar{z}_{\bm T^{\bs \lambda}, \bm T} \mid 
\bm T \in \st_{\bs \lambda}(\C_\ast) \right\}, \]
respectively. It is then clear that 
$\U_{\bs \lambda}$ is a $\Gamma^d A$-submodule of 
the $(\Gamma^d A, \Gamma^d B)$-bimodule 
$\J_{\bs \lambda} / \J_{\bs \lambda^+}$, and 
$\V_{\bs \lambda}$ is a $\Gamma^d B$-submodule. 
\end{definition}

The following analogue of 
Theorem \ref{thm:HK}  is the main result in this section. 

\begin{theorem}[Generalized Cauchy Decomposition]
\label{thm:gen_Cauchy}
Suppose $\bs \lambda \in \bs \L$.  
Then the map of $\k$-modules defined by  
\begin{align*}
\alpha_{\bs \lambda}: \, 
\J_{\bs \lambda} / \J_{\bs \lambda^+}& \, \to\,
\U_{\bs \lambda} \tensor \V_{\bs \lambda}: 
\ \  \bar{z}_{\bm S, \bm T}\ \mapsto\  
\bar{z}_{\bm S, \bm T^{\bs \lambda}}
\tensor \bar{z}_{ \bm T^{\bs \lambda}, \bm T},
\end{align*}
for all $(\bm S, \bm T) \in 
\st_{\bs \lambda}(\B_\ast) \times \st_{\bs \lambda}(\C_\ast)$, 
is an isomorphism of 
$(\Gamma^d A, \Gamma^d B)$-bimodules.  
The associated graded module of the generalized 
Cauchy filtration is thus given by 
\[\bigoplus_{\lambda\in \bs \L} 
\U_{\bs \lambda} \tensor \V_{\bs \lambda}.\]
\end{theorem}

\begin{proof}
Write $\phi_{\bs \lambda}: 
\J_{\bs \lambda} \to \F_{\bs \lambda, (\nu)}$ 
to denote the map obtained from   
$\hat{\phi}_\nu$ by restriction. 
There is an induced bimodule isomorphism 
\[ \bar{\phi}_{\bs \lambda} : 
\J_{\bs \lambda} / \J_{\bs \lambda^+}  \isoto
\F_{\bs \lambda, (\nu)} /\F_{\bs \lambda^+, (\nu)}\]
which follows from the definitions by using  
the decompositions  
$\J_{\bs \lambda}
= \J'_{\bs \lambda} \oplus \J_{\bs \lambda^+}$
and 
$\F_{\bs \lambda, (\nu)} = 
\F'_{\bs \lambda} \oplus \F_{\bs \lambda^+, (\nu)}$.

Hence by Proposition \ref{prop:gen_HK}, 
there is an isomorphism $\varphi_{\bs \lambda}$ 
making the upper right triangle commute in the following 
diagram 
\begin{equation}\label{eq:diagram}
\begin{tikzcd}[column sep= large, row sep=large]
\J_{\bs \lambda}/
\J_{\bs \lambda^+} 
\arrow[d, "\alpha_{\bs \lambda}"' , shift right =.1em]
\arrow[r, "{\bar{\phi}}_{\bs \lambda}"]
&
\F_{\bs \lambda, (\nu)} /
\F_{\bs \lambda^+, (\nu)}
\\
\U_{\bs \lambda} \tensor 
\V_{\bs \lambda}
&
W_{\bs \lambda}(U_\ast) \tensor 
W_{\bs \lambda}(V_\ast).
\arrow[u, "\bar{\psi}^{\bs \lambda}"', shift left = 0.1em]
\arrow[ul, "\varphi_{\bs \lambda}"']
\arrow[l, "\ \varphi'_{\bs \lambda} \tensor \varphi''_{\bs \lambda}"]
\end{tikzcd}
\end{equation}
In the bottom arrow, the map 
$\varphi'_{\bs \lambda}$ 
(resp.~$\varphi''_{\bs \lambda}$)  
denotes the homomorphism obtained by composing  
$\varphi_{\bs \lambda}$ with the embedding  
\[ W_{\bs \lambda} (U_\ast) \isoto  
W_{\bs \lambda} (U_\ast) \tensor \bar{y}_{\bm T^{\bs \lambda}}
\quad (\text{resp.}\ 
W_{\bs \lambda}(V_\ast) \isoto  
\bar{x}_{\bm T^{\bs \lambda}} \tensor W_{\bs \lambda}(V_\ast)).\]

In order to complete the proof, it suffices to show that the lower 
triangle in (\ref{eq:diagram}) is a commutative triangle of isomorphisms. 
For this, we compute: 
\begin{align*}
\varphi_{\bs \lambda}(\bar{x}_{\bm S} \tensor \bar{y}_{\bm T}) \, = 
&\  (\bar{\phi}_{\bs \lambda}^{-1} \circ \bar{\psi}^{\bs \lambda})
(\bar{x}_{\bm S} \tensor \bar{y}_{\bm T}) 
\\[.1cm]
= &\  \bar{\phi}_{\bs \lambda}^{-1} (\, \overline{
\psi^{\bs \lambda}
({x}_{\bm S} \tensor {y}_{\bm T})} 
\, )
 & \text{by Prop.~\ref{prop:gen_HK} }
\\[.1cm]
= &\  \overline{
(\phi'_{\bs \lambda})^{-1} \circ 
\psi^{\bs \lambda}
({x}_{\bm S} \tensor {y}_{\bm T}) }
 & \text{by (\ref{eq:gen_decomp})\, \text{and}\, (\ref{eq:sum_phi})}
\\[.1cm]
= &\ \bar{z}_{\bm S, \bm T}.
\end{align*}
It follows that  
$\varphi'_{\bs \lambda} \tensor \varphi''_{\bs \lambda}$ 
is an isomorphism since 
\[\varphi'_{\bs \lambda} \tensor \varphi''_{\bs \lambda}\, 
(\bar{x}_{\bm S} \tensor \bar{y}_{\bm T}) 
\, =\, 
\bar{z}_{\bm S, \bm T^{\bs \lambda}}
\tensor 
\bar{z}_{\bm T^{\bs \lambda}, \bm T} \]
for all $(\bm S, \bm T) \in 
\st_{\bs \lambda}(\B_\ast)\times \st_{\bs \lambda}(\C_\ast)$. 
Since it is now clear that the lower triangle is commutative, 
the proof is complete. 
\end{proof}

It follows from the proof of the theorem that 
$\U_{\bs \lambda}$ and $\V_{\bs \lambda}$ 
are each isomorphic to a respective (generalized) Weyl module. 
In the case $B= A^\op$, we  call  
$\U_{\bs \lambda}$ (resp.~$\V_{\bs \lambda}$) 
a left (resp.~right) Weyl submodule of 
the $\Gamma^d A$-bimodule 
$\J_{\bs \lambda}/\J_{\bs \lambda^+}$.

\section{Cellular Algebras}
Assume throughout this section 
that $\k$ is a noetherian integral domain.  
We first recall the definition of cellular algebras 
from  \cite{GL}, 
along with the reformulation given in  \cite{KX}. 
We then use the generalized Cauchy decomposition 
to describe a cellular structure on  
generalized Schur algebras $S^A(n,d)$. 
%

\subsection{Definition of cellular algebras}

\begin{definition}[Graham-Lehrer] 
An associative $\k$--algebra $A$ is called 
a {\em cellular algebra} with cell datum   $(I, M, C, \tau)$   
if the following conditions are satisfied:
\begin{itemize}
\item[(C1)] $(I, \trianglerighteq)$ is a finite partially ordered set. 
Associated to each $\lambda \in I$ is a finite set $M(\lambda)$.
The algebra $A$ has a $\k$-basis $C_{S,T}^\lambda$, where $(S,T)$ runs through 
all elements of $M(\lambda)\times M(\lambda)$ for all $\lambda \in I$.
\smallskip
\item[(C2)] 
The map  $\tau$ is an anti-involution of $A$ such that 
$\tau(C^{\lambda}_{S,T}) = C^\lambda_{T,S}$.
\smallskip
\item[(C3)] 
For each $\lambda \in I$ and $S,T \in M(\lambda)$ and each $a\in A$, 
the product $a C^\lambda_{S,T}$ can be written as 
$(\sum_{U\in M(\lambda)} r_a(U,S) C^{\lambda}_{U,T}) + r'$, where $r'$ is a linear 
combination of basis elements with upper index $\mu$ strictly larger than $\lambda$, and 
where the coefficients $r_a(U,S)\in \k$ do not depend on $T$.
\end{itemize}
\end{definition}

Let $A$ be  a cellular algebra with cell datum $(I, M, C,\tau)$. 
Given $\lambda \in I$,   it is clear that the set $J(\lambda)$  
spanned by the $C_{S, T}^{\mu}$ with $\mu  \trianglerighteq \lambda$   
is a $\tau$--invariant two sided ideal of $A$ (see \cite{GL}).
Let $J(\triangleright \lambda)$  denote the sum of ideals $J(\mu)$ with $\mu \triangleright \lambda$.
\smallskip

For $\lambda \in I$,  
the {\em standard module} $\Delta(\lambda)$ 
is defined as follows: as   a $\k$-module, $\Delta(\lambda)$  
is free with basis indexed by 
$M(\lambda)$, say  $\{C_{S}^\lambda\ |\ S \in M(\lambda)\}$;  
for each $a \in A$, 
the action of $a$ on $\Delta(\lambda)$ is defined by  
$ aC_{S}^\lambda=\sum_{ U} r_{ a}(U, S)   C_{U}^\lambda$ 
where the elements $r_{a}(U,S) \in \k$ are the coefficients  in (C3). 
Any left $A$-module isomorphic to $\Delta(\lambda)$ 
for some $\lambda$ will also be called a standard module. 
Note that for any $T \in M(\lambda)$, the assignment 
$C_S^\lambda \mapsto C^\lambda_{S,T} + J(\triangleright \lambda)$
defines an injective $A$--module homomorphism from 
$\Delta(\lambda)$ to $J(\lambda)/J(\triangleright\lambda)$.

\subsection{Basis-free definition of cellular algebras}\label{ss:KX}

In \cite{KX}, K\"onig and Xi provide an equivalent definition of cellular algebras 
which does not require specifying a particular basis.  
This definition can be formulated as follows.

\begin{definition}[K\"onig-Xi]
\label{def:KX} 
Suppose $A$ is a  $\k$-algebra with an anti-involution $\tau$. 
Then a two-sided ideal $J$ in $A$ is called a {\em cell ideal} if, and only if, $J = \tau(J)$ and there exists a left ideal $\Delta \subset J$ such that $\Delta$ is finitely generated and free over $\k$ and such  that there is an isomorphism of $A$-bimodules 
$\alpha: J \xrightarrow{\sim} 
\Delta \tensor \tau(\Delta)$ 
making the following diagram commutative:
\begin{equation*}
\begin{tikzcd}
J \arrow[r, "\alpha"] 
\arrow[d, "\tau\,"' ]  &  
\Delta \tensor \tau(\Delta)  
\arrow[d, "\,x\tensor y\, \mapsto\, \tau(y) \tensor \tau(x)"] 
\\
J \arrow[r, "\alpha"] 
& \Delta \tensor \tau(\Delta)
\end{tikzcd}
\end{equation*}
We say that a decomposition $A=J'_1\oplus \dots \oplus J'_r$ (for some $r$) 
into $\k$-submodules 
with $\tau(J'_j) = J'_j$ for each $j=1, \dots, r$ 
is a {\em cellular decomposition} of $A$ if 
setting $J_j:= \bigoplus_{1\leq i\leq j} J'_i$ gives a chain of 
($\tau$-invariant) two-sided ideals 
\[ 0=J_0 \subset J_1 \subset J_2 \subset \dots \subset J_r = A \] 
such that the quotient $J_j/J_{j-1}$ is a cell ideal 
(with respect to the anti-involution induced 
by $\tau$ on the quotient) of $A/J_{j-1}$.
\end{definition}

The above chain of ideals in $A$ is called a {\em cell chain}. 
For each ideal $J_j$ in a cell chain, 
we write 
\begin{equation}\label{eq:Delta_alpha}
\Delta_j \subset J_j/J_{j-1},\qquad
\alpha_j: J_j/J_{j-1} \equi 
\Delta_j \tensor \tau(\Delta_j) 
\end{equation}  
to denote the corresponding left ideal 
and $A$-bimodule isomorphism. 
Since $J_j = J'_j \oplus J_{j-1}$ for all $j$, 
we have a $\k$-module isomorphism 
$\alpha'_j: J'_j \cong \Delta_j\tensor \tau(\Delta_j)$
defined as the composition 
\begin{equation*}\label{eq:alpha}
\begin{tikzcd}[cramped, column sep=small]
\alpha'_j : J'_j \arrow[r, hook] 
& J_j \arrow[r, two heads]
& J_j/J_{j-1}
\arrow[r, " \alpha_j " ] 
&[3mm] \Delta_j\tensor \tau(\Delta_j).
\end{tikzcd}
\end{equation*}
It then follows by definition that we have a commutative diagram
\begin{equation}\label{eq:alpha2}
\begin{tikzcd}
J'_j
\arrow[r, "\alpha'_j"] 
\arrow[d, "i\,"' ]  
&  
\Delta_j \tensor_\k \tau(\Delta_j) 
\arrow[d, "\,x \tensor y\, \mapsto\, \tau(y) \tensor \tau(x)"] 
\\
J'_j
\arrow[r, "\alpha'_j"] 
& 
\Delta_j \tensor_\k \tau(\Delta_j)  
\end{tikzcd}
\end{equation}
of $\k$-module isomorphisms.

\begin{lemma}[K\"onig-Xi, \cite{KX}]\label{lem:KX}
Let $A$ be an associative $\k$-algebra with an anti-involution $\tau$. 
Then $A$ is a cellular algebra in the sense of \cite{GL} 
if and only if $A$ has a cellular decomposition.
\end{lemma}

\begin{proof}
We summarize the proof from \cite{KX}.  
Let $A$ be a cellular algebra with cell datum $(I, M, C, \tau)$. 
First, suppose $\lambda \in I$ is maximal. 
Then $J= J(\lambda)$ is a two-sided ideal by (C3) and $J=\tau(J)$ by (C2). 
Fix any element $T_\lambda \in M(\lambda)$. 
Define $\Delta$ as the $\k$--span of 
$C^\lambda_{S,T_\lambda}$ where $S$ varies.
Defining $\alpha$ by sending 
$C_{S,T_\lambda}^\lambda \tensor \tau(C_{T,T_\lambda}^\lambda)$
to $C_{S,T}^\lambda$ 
gives the required isomorphism.
Thus $J(\lambda)$ is a cell ideal.

Next, choose any enumeration $\lambda_1, \dots, \lambda_r$ of
 the elements of $I$ such that $i<j$ whenever 
 $\lambda_j \triangleright \lambda_i$. 
Set $J'_j\subset A$ (for each $j$) 
equal to the $\k$--span of all $C^{\lambda_j}_{S,T}$ (for varying $S,T$).
We have $\tau(J'_j) = J'_j$ by (C2). 
Since 
$J(\lambda_j) = J'_j \bigoplus J(\triangleright \lambda_j)$ for all $j$,
it follows that  $A= \bigoplus_j J'_j$ is a cellular decomposition.

For the converse, consider the index set 
 $I = \{1, \dots, r\}$ with the reversed 
 ordering 
 $1\triangleright \dots \triangleright r$. 
Choose a $\k$-basis $\{x^{(j)}_b\}_{b\in \B_j}$ 
of $\Delta_j$, for each $j\in I$.  
Setting $C^j_{b,c}\in J'_j$ 
to be the inverse image of 
$x^{(j)}_{b} \tensor \tau(x^{(j)}_{c})$ 
(for $b,c\in \B_j$) 
under  $\alpha'_j$ (for $j\in I$)
gives a $\k$-basis for $A$ of the form (C1). 
Since $\Delta_j$ is a left $A$-module, (C3) is satisfied.  
Finally, (C2) follows from the required commutative diagram
and the $\tau$-invariance of $J'_j$. 
It follows that  $\{C^j_{b,c}\}$ is a cellular basis. 
\end{proof}

From now on, we say that an algebra $A$ with anti-involution $\tau$ is {\em cellular} 
if either of the equivalent statements in Lemma \ref{lem:KX} is satisfied. 
The proof of the lemma shows that each ideal $\Delta_j$ (for $j=1, \dots, r$) 
for a cellular algebra $A$ is a standard module.

\subsection{Matrix algebras}\label{ss:Matrix}
Consider the matrix ring, $\Mat_n(\k)$, 
with matrix transpose, $\tr$, as  anti-involution. 
Let us write, $c: \Vn \tensor \Vn^\tr \isoto \Mat_n(\k)$, 
to denote the isomorphism mapping 
$v_i \tensor v_j^\tr \mapsto E_{ij}$ for all $i,j \in \set{1,n}$.
\smallskip

Now suppose $A$ is an algebra with anti-involution $\tau$, 
and let $J$ be a cell ideal with defining isomorphism 
$\alpha: J \isoto \Delta \tensor \tau(\Delta)$.
Then 
\[\Mat_n(J) := \Mat_n(\k)\tensor J\]
is a cell ideal of the matrix ring $\Mat_n(A)$ 
with respect to the anti-involution $\tr \tensor \tau$. 
The corresponding isomorphism is the map 
\[c^{-1}(\alpha): \Mat_n(J) \isoto  
\Vn(\Delta) \tensor \Vn^\tr(\tau(\Delta))\] 
defined by the composition 
\[\Mat_n(\k) \tensor J \xrightarrow{\ 
c^{-1}\hp{1}\tensor\hp{3} \alpha \ }
\left( \Vn\tensor  \Vn^\tr \right)  \tensor 
\left( \Delta \tensor \tau(\Delta) \right)
 \xrightarrow{\ \sim \ }
  \Vn\tensor \Delta \tensor \Vn^\tr \tensor \tau(\Delta).\]
More generally, we have the following.

\begin{lemma}\label{lem:Mat}
Suppose $A$ is a cellular algebra with anti-involution $\tau$ 
and cell chain $(J_j)_{j\in \set{1,r}}$.
Then the matrix ring $\Mat_n(A)$ is cellular with anti-involution 
$\tr \tensor \tau$ and cell chain
 $(\Mat_n(J_j))_{j\in \set{1,r}}$, 
where $\Mat_n(J_j) := \Mat_n(\k)\tensor J_j$ for all $j$.
\end{lemma}

\begin{proof}
It follows from the preceding paragraph 
 that the ideals, $\Mat_n(J_j)$,  
 form a cell chain, since
 $\Mat_n(J_j) / \Mat_n(J_{j-1}) \simeq \Mat_n(\k) \tensor (J_j / J_{j-1})$ 
 as $\Mat_n(A)$-bimodules. 
It is also clear that 
$\Mat_n(A)$ has a cellular decomposition 
\[\Mat_n(A) = \bigoplus \Mat_n(J'_j)\] 
where $A= \bigoplus J'_j$ denotes a corresponding 
cellular decomposition of $A$. 
\end{proof}

\subsection{Cellularity of generalized Schur algebras}
We now describe a cellular structure 
for  generalized Schur algebras $S^A(n,d)$. 
In this case, the generalized Cauchy filtration forms a cell chain, 
with the Weyl submodules from Theorem \ref{thm:gen_Cauchy} 
as standard modules.

\begin{theorem}\label{thm:cellular}
Suppose $A$ is a cellular algebra 
with anti-involution $\tau$.  
Then the generalized Schur algebra 
$S^A(n,d)$ is a cellular algebra, 
with respect to the anti-involution 
${\bs \tau}:= (\tr \tensor \tau)^{\tensor d}$,
for all $n,d \in \N$. 
\end{theorem}

\begin{proof}
If $A$ is cellular then so is $\Mat_n(A)$, by Lemma \ref{lem:Mat}. 
Since $S^A(n,d) = \G^d \Mat_n(A)$, 
it suffices to show that $\Gamma^d A$ is cellular, 
with respect to the anti-involution ${\bs \tau} = \tau^{\tensor d}$. 

Suppose that $A=J'_1 \oplus \dots \oplus J_r'$ 
is a cellular decomposition of $A$, 
with corresponding cell chain \[
0 = J_0 \subset J_1\subset \dots \subset J_r =A.\]  
For each $j\in\set{1,r}$, 
suppose $\{x^{(j)}_b\}_{b\in\B_j}$  and 
$\{y^{(j)}_b\}_{ b\in \B_j}$ are $\k$-bases
of $\Delta_j$ and  $\tau(\Delta_j)$, respectively, 
such that $y^{(j)}_b:= \tau(x^{(j)}_b)$ for all $j$, 
and let $\Delta_j$ and $\alpha_j$ 
be as in (\ref{eq:Delta_alpha}).

Considering $\bs \Lambda = \L^+(\sharp \B_1, \dots, \sharp \B_r)$ 
as a totally ordered subset of $\L^+_d(\N)^r$ 
by restricting the order $\preceq$ in Definition \ref{lex2}, 
it follows from Lemma \ref{lem:gen_decomp}
and Remark \ref{rmk:multi_tab} 
that we have decompositions 
\begin{equation}\label{eq:cell_decomp}
\Gamma^d J = 
\bigoplus_{\bs \lambda \in \bs \Lambda} \J'_{\bs \lambda}, 
\quad \text{and} \quad
\J_{\bs \lambda} = 
\bigoplus_{\bs \mu\succeq \bs \lambda} \J'_{\bs \mu} 
\quad \text{for each } \bs \lambda \in \bs \Lambda, 
\end{equation}
since $\J'_{\bs \lambda} = 0$ if  
$\bs \lambda \notin \bs \Lambda$.

Notice that  ${\bs \tau} = \tau^{\tensor d}$ coincides with the map 
$\Gamma^d(\tau): \Gamma^d A\to \Gamma^d A$
induced by the functor $\Gamma^d$. 
To complete the proof, 
we need to show that the left-hand side of (\ref{eq:cell_decomp}) gives a 
cellular decomposition of $\Gamma^d A$ 
with respect to this anti-involution. 

Let $\Delta_{\bs \lambda}$ be the 
the left Weyl submodule  
$\U_{\bs \lambda} \subset \J_{\bs \lambda}/ \J_{\bs \lambda^+}$ 
of Theorem \ref{thm:gen_Cauchy}. 
Then it remains to check the following hold
for each $\bs \lambda \in \bs \Lambda$: 
\begin{enumerate}[(i)]
\item 
${\bs \tau}(\J'_{\bs \lambda}) = \J'_{\bs \lambda}$, 
\medskip

\item 
${\bs \tau}(\Delta_{\bs \lambda}) = 
\V_{\bs \lambda }$,
\medskip

\item 
$\J_{\bs \lambda} / \J_{\bs \lambda^+}$ is a cell ideal. 
\smallskip
\end{enumerate}
Assuming (i) and (ii) hold for each $\bs \lambda$, (iii) 
will follow from the commutativity of the diagram  
\begin{equation*}
\begin{tikzcd}
\J_{\bs \lambda} / \J_{\bs \lambda^+}
\arrow[r, "\alpha_{\bs \lambda}"] 
\arrow[d, "\tau\,"' ]  
&  
\Delta_{\bs \lambda} \tensor \tau(\Delta_{\bs \lambda})  
\arrow[d, "\,x\tensor y\, \mapsto\, \tau(y)\tensor \tau(x)"] \\
\J_{\bs \lambda} / \J_{\bs \lambda^+}
\arrow[r, "\alpha_{\bs \lambda}"] 
& 
\Delta_{\bs \lambda} \tensor \tau(\Delta_{\bs \lambda})  
\end{tikzcd}
\end{equation*}
where $\alpha_{\bs \lambda}$ 
is the $\Gamma^d A$-bimodule isomorphism 
from Theorem \ref{thm:gen_Cauchy}.

Now fix $\bs \lambda \in \bs \Lambda$, 
and set $\nu = |\bs \lambda|$. 
Then 
$\J'_{\bs \lambda}$, $\Delta_{\bs \lambda}$, and $\J_{\bs \lambda}/ \J_{\bs \lambda^+}$  
have $\k$-bases given by the sets 
\[\left\{ z_{\bm S, \bm T } \mid  
\bm S, \bm T \in \st(\B_\ast) \right\}, 
\qquad
\left\{\bar{z}_{\bm S, \bm T_{\bs \lambda}} 
\mid \bm S  \in \st(\B_\ast) \right\}, 
\qquad 
\left\{\bar{z}_{\bm S, \bm T} \mid  
\bm S, \bm T \in \st(\B_\ast) \right\}\]
respectively, where 
$z_{\bm S,\bm T}\in \J'_{\bs \lambda}$ is defined in 
 \eqref{z_basis}. 
It follows that each of the conditions (i)-(iii) will be satisfied 
provided that 
$\tau(z_{\bm S, \bm T}) = z_{\bm T, \bm S}$ 
for all $\bm S, \bm T\in \st(\B_\ast)$. 

We claim that the following diagram is commutative: 
\[\begin{tikzcd}[row sep = huge, column sep = 4.3em]
\Gamma^{\bs \lambda}(\Delta_\ast) \tensor \Gamma^{\bs \lambda}(\tau(\Delta_\ast))
\arrow[r, "\psi^{\bs \lambda}" ]
\arrow[d, "\, \tw\hs{1pt} \circ {\left( \Gamma^{\bs \lambda}(\tau)\, \tensor\, \Gamma^{\bs \lambda}(\tau)\right)} " , shift right=0.1em ]
 &[-1em] \Gamma^{(\nu)}(\Delta_\ast {\tensor} \tau(\Delta_\ast))
\arrow[d, "\, \Gamma^{(\nu)}\left(\tw\hs{1pt}\circ{\left( \tau\, \tensor\, \tau \right ) }  \right) " , shift left=0.75em  ]
& \Gamma^{(\nu)}(J'_\ast) \arrow[l, "\ \Gamma^{(\nu)}(\alpha'_\ast) "' ]
\arrow[r, " \nabla^\nu" ] \arrow[d, "\, \Gamma^{(\nu)}(\tau) " ] &[-1em]
J'_\nu \arrow[d, "\, \tau " ] \\
\Gamma^{\bs \lambda}(\Delta_\ast) \tensor \Gamma^{\bs \lambda}(\tau(\Delta_\ast))
\arrow[r, "\psi^{\bs \lambda}" ]
& \Gamma^{(\nu)}(\Delta_\ast { \tensor} \tau(\Delta_\ast))
& \Gamma^{(\nu)}(J'_\ast)
\arrow[l, "\ \Gamma^{(\nu)}(\alpha'_\ast) "' ]
\arrow[r, " \nabla^\nu" ]
& J'_\nu, \end{tikzcd} \]
with the first (middle) vertical map(s) induced by the action of   
$\Gamma^{\bs \lambda}$ (resp.~$\Gamma^{(\nu)}$)
considered as a functor $\P_\k^{\times r} \to \P_\k$. 
The commutativity of the left-hand square 
can be checked 
using the definition of $\psi^{\bs \lambda}$ together with 
Lemma \ref{psi}.
The commutativity of the middle square  
follows from the functoriality 
of $\Gamma^{(\nu)}$ 
and diagram (\ref{eq:alpha2}). 
Finally, the commutativity of the right-hand square follows by Lemma \ref{lem:natural}. 
We thus have $\tau(z_{\bm S, \bm T})= z_{\bm T, \bm S}$ 
for all $\bm S, \bm T\in \st(\B_\ast)$, and the proof is complete. 
\end{proof}

Let us write 
 $\bs \Lambda^{\mathrm{op}}$ to denote the set 
$\bs \Lambda$ with opposite total ordering. 
Then it follows from the  above proofs of Lemma \ref{lem:KX}
and Theorem \ref{thm:cellular} 
that the set
\[\big\{ z_{\bm S, \bm T} \mid 
{\bs \lambda} \in \bs \Lambda^{\mathrm{op}},\  
{\bm S}, {\bm T} \in \st_{\bs \lambda}(\B_1,\dots, \B_r)  \big\}.\]
is a cellular basis for $\G^d A$. 
A corresponding cellular basis for $S^A(n,d)$ can be obtained in 
a similar way, 
by replacing $A$ by $\Mat_n(A)$. 
\smallskip

In the next example, we describe an explicit cellular basis 
for a special case of a 
generalized Schur algebra of the form $S^Z(n,d)$,  
where $Z$ is a zig-zag algebra. 
We essentially follow the definition in \cite{KM3}, 
using slightly different notation. 
Note also that we only consider $Z$ as an  
ordinary non-graded algebra, 
rather than a $\Z/2$-graded superalgebra 
as in \cite{KM3}. 

\begin{example}[Zig-zag algebra]\label{ex:zig}
We consider the zig-zag algebra 
associated to 
the quiver below.  
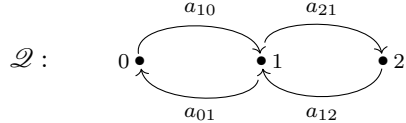
\[\mathscr Q : \qquad 
\begin{tikzpicture}[
baseline=-2pt, black,line width=1pt, scale=0.4,
every node/.append style={font=\fontsize{8}{8}\selectfont}   ]%
\coordinate (0) at (0,0);
\coordinate (1) at (4,0);
\coordinate (2) at (8,0);
\draw [thin, black,->,shorten <= 0.1cm, shorten >= 0.1cm]   (0) to[distance=1.5cm,out=100, in=100] (1);
\draw [thin,black,->,shorten <= 0.25cm, shorten >= 0.1cm]   (1) to[distance=1.5cm,out=-100, in=-80] (0);
\draw [thin, black,->,shorten <= 0.25cm, shorten >= 0.1cm]   (1) to[distance=1.5cm,out=80, in=100] (2);
\draw [thin,black,->,shorten <= 0.1cm, shorten >= 0.1cm]   (2) to[distance=1.5cm,out=-100, in=-80] (1);
\draw(0,0) node{$\bullet$};
\draw(4,0) node{$\bullet$};
\draw(8,0) node{$\bullet$};
\draw(0,0) node[left]{$0$};
\draw(4,0) node[right]{$1$};
\draw(8,0) node[right]{$2$};
\draw(2,1.2) node[above]{$a_{10}$};
\draw(6,1.2) node[above]{$a_{21}$};
\draw(2,-1.2) node[below]{$a_{01}$};
\draw(6,-1.2) node[below]{$a_{12}$};
  \end{tikzpicture}\]
Recall from \cite[Section 7.9]{KM3} that the {\em extended zig-zag algebra}, $\tilde{Z}$, 
is defined in this case as the quotient 
of the path algebra $\k \mathscr Q$ modulo the following 
relations: 
\begin{enumerate}
\item All paths of length three or greater are zero.
\item All paths of length two that are not cycles are zero.
\item All length-two cycles based at the same vertex are equivalent.
\item $ a_{21} a_{12}=0$.
\end{enumerate}
The length zero paths are denoted 
$e_0, e_1, e_2$
and correspond to standard idempotents,
with 
$e_i a_{ij} e_j = a_{ij}$ for all admissible $i,j$. 
Let $e:= e_0 + e_1 \in \tilde{Z}$. 
Then the corresponding {\em zig-zag algebra} is $Z:=e \tilde{Z} e \subset \tilde{Z}$. 
Then $Z$ is a cellular algebra, 
with anti-involution 
defined by 
$\tau(e_i)=e_i$ and $\tau(a_{ij}) = a_{ji}$ for all $i,j$. 

Let us describe a corresponding cellular decomposition. 
First let 
\[
x_1:=a_{12},\ \ 
x_2:=e_1,\ \ 
x_3:=a_{01},\ \ 
x_4:=e_0.
\]
and set $y_i:= \tau(x_i)$, for
$i\in \set{1,4}$.
Then we have corresponding sets 
\[
X(1):=\{x_1\},
\quad
X(2):=
\{x_2,x_3\},  
\quad
X(3):=
\{x_4\}, 
\]
and 
\[
Y(1):=\{y_1\},
\quad
Y(2):=
\{y_2,y_3\},  
\quad
Y(3):=
\{y_4\},
\]
parametrized by the totally ordered sets 
$\B_1:= \{1\}$, 
$\B_2:= \{2<3\}$, and 
$\B_3:= \{4\}$, 
respectively. 
We may then define a cellular decomposition 
\[Z= J_1' \oplus J_2'  \oplus J_3', \]
where 
$J_j' := \text{span}\{ xy \mid x\in X(j), y\in Y(j) \}$, 
for $j\in \set{1,3}$.

Now let $\bs \L^{\mathrm{op}}$ denote the set  
$\bs \L = \L_3^+(1,2,1)$
 with the opposite total ordering. 
Then one may then check using formula \eqref{z_basis}
and the  proof of Lemma \ref{lem:KX} 
that $S^Z(1,2) = \Gamma^2 Z$ has the cellular basis 
described 
in the table below, 
where 
$\bs \lambda$ runs through all multipartitions in the set 
$\bs \L^{\mathrm{op}}$,
and where 
$\mathbf S$, $\mathbf T$ denote 
standard multitableaux of shape $\bs \lambda$, respectively. 
\smallskip
 \begin{center}
 \hspace*{-1cm}
  \begin{tabular}{|c|  c c| c | }
\hline
\rule{0pt}{1.0\normalbaselineskip}
$\boldsymbol{\lambda}$  & $\mathbf{S}$ & $\mathbf{T}$ &  
$z_{\mathbf{S,T}}$ 
\\[.1cm]
\hline 
\rule{0pt}{1.25\normalbaselineskip}
$(\o,\o,(2))$ & 
$( \o, \o, \scriptsize\young(44))$ & $( \o, \o, \scriptsize\young(44))$ & 
\ $e_0^{\otimes 2}$ 
\\[.3cm]
\hline \rule{0pt}{1.25\normalbaselineskip}
$(\o,(1),(1))$ & 
$(\o, \scriptsize\young(2),\, \young(4))$ & $(\o, \scriptsize\young(2),\, \young(4))$ & 
$e_0 \ast e_1$
\\[.25cm]
 &  $''$ & 
$(\o, \scriptsize\young(3),\, \young(4))$ & 
$e_0 \ast a_{10}$
\\[.25cm]
& $(\o, \scriptsize\young(3),\, \young(4))$ & $(\o, \scriptsize\young(2),\, \young(4))$ & 
$e_0 \ast a_{01}$
\\[.25cm]
&  $''$ & $(\o, \scriptsize\young(3),\, \young(4))$ & 
$e_0 \ast ( a_{01}a_{10})$
\\[.3cm]
\hline \rule{0pt}{1.25\normalbaselineskip}
$(\o,(1,1),\o)$ & 
$(\o, \scriptsize\young(2,3), \o)$ & $(\o, \scriptsize\young(2,3), \o)$ & 
 $e_1 \ast ( a_{01}a_{10})$
\\[.3cm]
\hline \rule{0pt}{1.25\normalbaselineskip}
$(\o,(2),\o)$ & 
$(\o, \scriptsize\young(22), \o)$ & $(\o, \scriptsize\young(22), \o)$ & 
\ $e_1^{\otimes 2}$ 
\\[.25cm]
 & $''$  & $(\o, \scriptsize\young(23), \o)$ &
$ e_1 \ast a_{10}$
\\[.25cm]
 & $''$ & $(\o, \scriptsize\young(33), \o)$ &
 $a_{10}^{\otimes 2}$ 
\\[.25cm]
& 
$(\o, \scriptsize\young(23), \o)$ & $(\o, \scriptsize\young(22), \o)$ & 
$ e_1 \ast a_{01}$
\\[.25cm]
 & $''$ & $(\o, \scriptsize\young(23), \o)$ &
 $ e_1 \ast ( a_{01}a_{10})+ a_{10} \ast a_{01}$
\\[.25cm]
 & $''$ & $(\o, \scriptsize\young(33), \o)$ &
$a_{10} \ast a_{01}a_{10}$
\\[.25cm]
&
$(\o, \scriptsize\young(33), \o)$ & $(\o, \scriptsize\young(22), \o)$ &
$a_{01}^{\otimes 2}$ 
\\[.25cm]
 & $''$ & $(\o, \scriptsize\young(23), \o)$ &
$a_{01} \ast (a_{01}a_{10})$
\\[.25cm]
 & $''$ & $(\o, \scriptsize\young(33), \o)$ &
$( a_{01} a_{10})^{\otimes 2}$ 
\\[.25cm]
\hline \rule{0pt}{1.25\normalbaselineskip}
\rule{0pt}{1.25\normalbaselineskip}
$((1),\o,(1))$ & 
 $({\scriptsize\young(1)}, \o, {\scriptsize\young(4)})$ &  $({\scriptsize\young(1)}, \o, {\scriptsize\young(4)})$
 &
$(a_{12} a_{21})\ast e_0$ 
\\[.3cm]
\hline \rule{0pt}{1.25\normalbaselineskip}
$((1),(1),\o)$ & 
 $(\scriptsize\young(1), \young(2), \o)$ & $(\scriptsize\young(1), \young(2), \o)$
&
$e_1 \ast (a_{12} a_{21})$ 
\\[.25cm]
 & $''$ &  $(\scriptsize\young(1), \young(3), \o)$ &
$a_{10} \ast (a_{12} a_{21})$ 
\\[.25cm]
&  $(\scriptsize\young(1), \young(3), \o)$ &  $(\scriptsize\young(1), \young(2), \o)$ &
$a_{01} \ast (a_{12} a_{21})$ 
\\[.25cm]
& $''$  &  $(\scriptsize\young(1), \young(3), \o)$ &
$(a_{01}  a_{10}) \ast (a_{12} a_{21})$ 
\\[.3cm]
\hline \rule{0pt}{1.25\normalbaselineskip}
$((2),\o,\o)$ & 
 $({\scriptsize\young(11)}, \o, \o)$ & $({\scriptsize\young(11)}, \o, \o)$
&
$(a_{12} a_{21})^{\otimes 2}$ 
\\[.35cm]
\hline
\end{tabular}
\hspace*{-1cm}
\end{center}
\medskip
The symbol, $\o$, is used above to denote 
an empty partition or tableau, respectively, 
and the symbol $''$ denotes a repeated item  
from the above entry.

\end{example}

\subsection{Cellularity of wreath products  $A\wr \Si_d$}
Let us first recall a result of \cite{KX} concerning idempotents fixed by 
an anti-involution. 
\begin{lemma}[\cite{KX}]\label{lem:KX2}
Let $A$ be a cellular algebra with anti-involution $\tau$. 
If $e\in A$  is an idempotent fixed by $\tau$, 
then the algebra $eAe$ is cellular with respect to the 
restriction of $\tau$. 
\end{lemma}

We then have the following consequence of Theorem \ref{thm:cellular}, 
which is obtained via generalized Schur-Weyl duality. 

\begin{corollary}\label{wreath}
Suppose $d\in \N$.  If $A$ is a cellular algebra, 
then $A\wr \Si_d$ is also cellular.
\end{corollary}

\begin{proof}
Fix some $n \geq d$.   Write $S^A = S^A(n,d)$,  
and let $e \in S^A$ 
denote the idempotent $e:=\xi_\omega$. 
It then follows by Proposition \ref{prop:EK}.(ii) that 
there is an algebra isomorphism 
$A\wr \Si_d \cong e \hp{3} S^A \hp{1} e$. 
Since 
\[\bs \tau(e) = 
(E_{1,1})^\tr \ast \cdots \ast (E_{d,d})^\tr = e,\]
the cellularity of $A\wr \Si_d$
follows from Theorem \ref{thm:cellular} and Lemma \ref{lem:KX2}. 
\end{proof}

Since the above result holds for an arbitrary cellular algebra $A$, 
we thus obtain an alternate proof of the main results of \cite{GG} and \cite{RGr} 
mentioned in the introduction. 

\end{document}